\documentclass[a4paper,final]{article}
\usepackage{amsmath}
\usepackage{amsthm}
\usepackage{amssymb}
\usepackage{graphicx}
\usepackage{multirow}
\usepackage{wrapfig}
\usepackage{float}
\usepackage{xcolor}
\usepackage{hyperref}
\usepackage{setspace}
\usepackage{mathtools}
\usepackage{algorithm}
\usepackage{algorithmic}
\renewcommand{\algorithmiccomment}[1]{\bgroup\hfill//~#1\egroup}
\usepackage{eurosym}
\usepackage[utf8]{inputenc}
\usepackage[english]{babel}
\usepackage{amsfonts}
\usepackage{authblk}
\usepackage{xcolor}
\usepackage[normalem]{ulem}
\newtheorem{theorem}{Theorem}

\newtheorem{lemma}[theorem]{Lemma}

\title{Logic-Based Benders for intermodal operations with delay penalties}

\author[1,*]{Ioannis Avgerinos}
\author[1]{Ioannis Mourtos}
\author[1]{Georgios Zois}
\affil[1]{\small ELTRUN Lab, Department of Management Science and Technology, Athens University of Economics and Business,\protect\\ \small Patision 76, 10434, Athens, Greece E-mail: {\tt iavgerinos@aueb.gr, mourtos@aueb.gr, georzois@aueb.gr}}

\affil[*]{Corresponding Author} 
\date{} 

\begin{document}
\maketitle

\begin{abstract}
Intermodal logistics typically include the successive stages of intermodal shipment and last-mile delivery. We investigate this problem under a novel Logic-Based Benders Decomposition, which exploits the staged nature of the problem to minimise  the sum of transport costs and delivery penalties. We establish the validity of our decomposition and apply effective optimality cuts. Apart from models and formal proofs, we provide extensive experimentation on random instances of considerable scale that shows the improvement achieved in terms of small gaps and shorter time compared to a monolithic MILP approach. Last, we propose a major extension of our generic method for a real logistics case. The implementation of the extension on real instances show the versatility of our method in terms of supporting different planning approaches thus leading to actual cost improvements. 
\end{abstract}

\textbf{Keywords.} Logistics, Benders decomposition, intermodal transportation. 

\section{Overview and Background}
The operation of functional distribution networks remains a major challenge in freight transportation, as indicated also by the COVID-19 outbreak, which required acute response to severe disruptions. A report of the International Finance Corporation of the World Bank Group \cite{IFC21} lists the closure of the borders between Germany and Poland and the shortage of truck drivers in India as two events that challenged the endurance of the logistics' providers. The increase of railway freight transports, as a response to the radical decrease of air freight capacity between China and Europe during the pandemic \cite{IFC21}, shows that using alternative modes of transport could tackle such challenges. In parallel, \emph{Direct Supply Chains} \cite{Men11} aim at dynamically shaping \textit{Supplier - Organization - Customer} routes that minimize both the operating costs and the delayed deliveries. As the \textit{last-mile} stage of the chain, in which the delay penalties are involved, is determined by the intermodal transports of the preceding \textit{middle-mile} stage, a proper coordination between these two stages is required.

In this study, we aim at minimizing the total transportation costs and delivery penalties of a logistics supply chain by an exact optimization method that finds provably near-optimal solutions in realistically-sized instances in reasonable time. For that to happen, the intermodal transportation and the last-mile deliveries must be simultaneously planned as heavier transport costs occur in the former stage, whereas delays are realised in the latter. A Mixed-Integer Linear Programming (MILP) model is intractable for a commercial Integer Programming solver and this is where our technical contribution comes forward. That is, we propose a Logic-Based Benders' decomposition (LBBD) \cite{Hook03, Hook07} of the two-staged problem, motivated by the recent LBBD literature on supply chain problems (e.g., \cite{Alfandari21, Taherkhani20}). In our LBBD, the so-called \emph{master} problem handles the intermodal transportation problem, whereas a set of \emph{subproblems} handles the last-mile delivery. Further aspects of interest include the introduction of appropriate optimality cuts plus subproblems' relaxation and valid inequalities, all aiming at strengthening the master problem's formulation. In particular, inspired by a modern concept of decomposition schemes, namely \emph{Partial Benders decomposition}, we retain information of the subproblems to the master problem, so that a faster convergence is achieved.

\textbf{Relevant Background.} \emph{Cross-docking} is a logistics concept which includes consolidation points for the orders, as intermediate stations between the supplying points and the customers. A comparison of the cost-effectiveness of Cross-Docking and the regular direct-shipment policy is provided by \cite{aueb17}, and a model that combines both approaches is described by \cite{Azi20}. Since Cross-Docking operations occur in our setting, we draw ideas from \cite{Yu16,Qiu21}. An extended survey on the inclusion of intermediate facilities in freight transport networks appears in \cite{Gua16}. If the network includes multiple Cross-Docking stations, the multi-depot Location-Routing problem of \cite{Wu02} and the multi-depot VRPTW of \cite{Li14} would be of value.

Intermodal networks have long been studied under an optimisation viewpoint \cite{Mach04}. This is not surprising because they include several alternative topologies \cite{Wox07} and network designs. Notably, our work considers a \textit{hub-and-spoke} layout, that is the direct shipment of orders from hubs to customers. In addition to the selection of the means of transport that will perform the shipment, a logistics supply chain must also consider the delays of the deliveries: indicatively \cite{Ish12} considers fixed arrival rates and transport times to compute the delays of the deliveries in a USA road-rail intermodal network. An intermodal network typically includes timetabled transport services of different modes (rail or ship), in which each service can be performed during multiple time windows in a planning horizon. As this is also the case in our study, we extend the approach of \cite{Moc11} that constructs identical copies of transport services that depart in different time instances. Time aspects are important also for the hub location problem in an intermodal network \cite{Alu12}.

Recently, \cite{Gha19} offer a heuristic for transferring hazardous materials through a network of railway or roadway nodes, thus solving a routing problem under stochastic scenarios. \cite{Fazi20} aim to connect inland hubs with ports through a network under strict time windows, so that a set of containers is shipped by the sea terminals. While the model of \cite{Fazi20} considers containers that are already loaded, the study of \cite{Li21} includes hubs that are also charged with the consolidation of orders to vehicles before the latter are routed through a multi-modal network. All three studies examine realistic instances that are significantly larger than the ones previously examined in the literature. Therefore, these papers suggest heuristic methods, since a commercial solver cannot respond in reasonable time. We are quite motivated by these developments, as our approach can handle instances of analogous size and of broader scope.

\begin{table}[tbh]
\centering
\resizebox{\textwidth}{!}{
\begin{tabular}{lll}
\hline
\textbf{Reference}            & \textbf{Description}                              & \textbf{Solution Approach}                         \\ \hline
Wu et al., 2002 \cite{Wu02}               & Multi-depot Vehicle Routing Problem              & Heuristic, TS                               \\
Rousseau et al., 2004 \cite{Mach04}         & Vehicle Routing Problem with Time Windows         & CG, CP            \\
Moccia et al., 2011 \cite{Moc11}          & Intermodal Transportation Problem                 & CG, BnC                    \\
Alumur et al., 2012 \cite{Alu12}           & Intermodal Hub Location Problem                   & MIP                                                  \\
Ishfaq \& Sox, 2012 \cite{Ish12}            & Intermodal Network Design                         & Heuristic, TS                               \\
Li et al., 2014 \cite{Li14}               & Multi-depot Vehicle Routing Problem               & GA, LS                      \\
Ghane-Ezabadi \& Vergara, 2016 \cite{Gha16} & Intermodal Network Design                         & Decomposition                                        \\
Yu et al., 2016 \cite{Yu16}               & Vehicle Routing Problem with Cross-Docking        & MIP, SA                             \\
Nikolopoulou et al., 2017 \cite{aueb17}     & Cross-Docking and Direct Shipment Comparison      & LS                                         \\
Wang \& Meng, 2017 \cite{Wang17}             & Intermodal Network Design                         & Nonlinear Algorithm, BnB                \\
Ghaderi \& Burdett, 2019 \cite{Gha19}            & Intermodal Transportation Problem                            & Heuristic                                \\
Mahéo et al., 2019 \cite{Mah19}            & Transit Network Design                            & BD                                \\
Azizi \& Hu, 2020 \cite{Azi20}              & Pickup and Delivery Problem with Location Routing & MIP, BD                           \\
Belieres et al., 2020 \cite{Bel20}              & Logistics Network Design & BD                           \\
Fazi et al., 2020 \cite{Fazi20}              & Pickup and Delivery Problem with Location Routing & Heuristic, TS                           \\
Avgerinos et al., 2021 \cite{Avg21}        & Intermodal Transportation Problem                 & BD                                \\
Fontaine et al., 2021 \cite{Fon21}        & Intermodal Transportation Problem                 & BD                                \\
Li et al., 2021 \cite{Li21}        & Intermodal Transportation Problem                 & Heuristic                                \\
Qiu et al., 2021 \cite{Qiu21}             & Vehicle Routing Problem with Cross-Docking        & BnC                                  
  \\ \hline \\
TS & Tabu Search & \\
CG & Column Generation & \\
CP & Constraint Programming & \\
BnC & Branch-and-Cut & \\
BnP & Branch-and-Price & \\
MIP & Mixed Integer Problem & \\
LS & Local Search & \\
GA & Genetic Algorithm & \\
SA & Simulated Annealing & \\
BnB & Branch-and-Bound &\\
BD & Benders Decomposition &
\end{tabular}
}
\caption{Summary of literature on Routing and Intermodal Transportation Problems}
\label{tab:literature}
\end{table}

The literature on intermodal transportation is enormous, yet the one summarized in Table \ref{tab:literature} shows the use of multiple optimization methods. Although heuristic methods remain of great value, the contemporary landscape reveals that 3PL provider compete on total cost, thus finding the optimal solution matters. Exact methods are known to be slow or inappropriate in terms of memory requirements, thus our work exploit the staged nature of our problem to apply a decomposition scheme. Decomposition-based methods are widely used in transportation problems, some of them being intermodal \cite{Fon21, Gha16}. The main framework of our exact approach for this paper will be Benders Decomposition \cite{Ben62} and, in fact, its logic-based extension \cite{Hook03} (LBBD). The closest study relying on LBBD is \cite{Mah19}, although LBBD has broad applicability \cite{Rah17}. Beyond modelling, the implementation of LBBD is a non-trivial task as its success depends on the integration of valid cuts and also the combination of MILP and Constraint Programming (CP) formulations, as nicely discussed in \cite{Nutmeg20}. Recently, the concept of \emph{Partial Benders Decomposition}, that is addition of information derived from the subproblem to strengthen the master problem, has been implemented on stochastic network design problems \cite{Cra21}. As \cite{Bel20} and \cite{Fon21} also did for their presented logistics network design problems, we implement this idea on our decomposition setting.

\textbf{Our contribution.} To the best of our knowledge, this is the first paper offering an exact approach that considers together the intermodal transportation stage and the last-mile delivery of a Direct Supply Chain. Unifying these two stages allows us to consider the delayed delivery penalties, which are significantly determined by the preceding intermodal transportation, but nevertheless charged during the last-mile delivery.

From an application-oriented standpoint, we experiment with large-scale randomly generated instances, with $60-160$ orders. These orders are transferred through a network of $7$ hubs, connected with hundreds of intermodal services. Let us note that \cite{Gha16} solve instances of $150$ orders, consolidated in $4$ hubs and transferred by combinations of $3$ modes per connected pair of hubs, while \cite{Li21} address the shipment of $300$ orders to a network of $20$ nodes; \cite{Fazi20} consider datasets of $600$ containers, nevertheless their network is unimodal. Thus, our instances are comparable with the largest in the intermodal transportation literature. In addition, we expand the scope of the typical intermodal network planning problem, by considering penalties of delays in the objective function.

To evaluate the efficiency of our method under the additional complications imposed by a real-life application, we extend our formulations to an actual case of a 3PL provider in Europe. This case integrates all three stages of a logistics supply chain (i.e., \textit{first-mile}, \textit{middle-mile}, \textit{last-mile}). For the last-mile stage, two alternative configurations are used. Our method exploits the fact that the practical restrictions of the problem allow us to pre-compute fixed routes for the first and last-mile stages, thus modelling the corresponding vehicle routing problems as exact-covering or resource-constrained scheduling problems. The extended LBBD algorithm is evaluated, after being implemented on real instances, given by the logistics provider. This paper is a strongly enhanced generalisation of \cite{Avg21}, whose experiments on smaller instances show already that a commercial solver cannot handle the entire problem thus making the decomposition necessary.

\section{Problem definition and notation} \label{Section:Model}
\subsection{Problem definition}
Let $I$ be the set of nodes of an intermodal network. Subset $J\subset I$ includes the \emph{Distribution Centers hubs} (\emph{DC hubs}), in which orders $N$ are consolidated into loading compartments $G$. Each compartment is dedicated to one DC hub, hence set $G$ is divided into subsets $G_{j}, j\in J$. Subset $L\subset I$ includes \emph{Satellite} hubs, in which orders are brought to before their final destination. The rest of the nodes of set $I$, namely \emph{intermediate hubs}, are labelled as $K\subset I$.

After the completion of the consolidation process, all loaded compartments must be shipped to any $Satellite$ hub $l\in L$, via an intermodal services network. We consider three transport modes (i.e., roadway, railway, and seaway modes) and a set of scheduled transport services, denoted by $S$. Each service $s\in S$ is featured with an origin node $origin_{s}\in I$, a destination node $dest_{s}\in I$, a fixed time of departure and arrival ($departure_{s}$ and $arrival_{s}$ respectively), a variable cost $c^{travel}_{s}$, which is charged for each compartment that is shipped by service $s$, a fixed cost $c^{fixed}_{s}$, and a capacity (in number of compartments) $Q_{s}$. For each node $i \in I,$ we distinguish the services arriving in and departing from it, by introducing the sets $\delta^-(i)$ and $\delta^+(i)$, respectively. Last, each service is linked with a subset of invalid co-assignments $F_{s}\subset S$, that is the set of services that cannot be added in the same intermodal route with $s$. Each subset $F_{s}$ includes all services $p\in S\setminus \{s\}$ that:
\begin{description}
\item{$a.$} depart or arrive during service $s$ (i.e., $departure_{s} \leq departure_{p} < arrival_{s}$ or $departure_{s} < arrival_{p} \leq arrival_{s}$),
\item{$b.$} start from the destination of $s$ before the completion of $s$ (i.e., $origin_{p} = dest_{s}$ and $departure_{p} < arrival_{s}$), or
\item{$c.$} arrive to the origin of $s$ later than the departure of $s$ (i.e., $dest_{p} = origin_{s}$ and $arrival_{p} > departure_{s}$).
\end{description}
The solution of the intermodal transportation stage is a set of combined services, that start from a \textit{DC hub} and terminate to a \textit{Satellite hub}. Each service is linked with a set of assigned compartments. Note that, despite the dedication of each order or compartment to a \textit{DC hub}, there are no pre-assignments to \textit{Satellite hubs}, hence each loaded compartment could be transferred to any \textit{Satellite hub} by eligible services, regardless of its components.

As the problem requires the assignment of loaded compartments to services of different capacity and cost, the intermodal transportation stage is structured as a Generalized Assignment Problem, which is known to be NP-hard \cite{GAP86}. The complexity arises from orders being loaded to compartments, which must then be transferred through the intermodal network at a minimum `flow' cost.

Concerning now the last-mile stage, each order $n\in N$ is featured with a deadline $t_{n}$ and a value of weight $w_{n}$, which is related with the importance of the order. The deadlines are not strict, nevertheless their violation implies a penalty cost, equal with the number of days of delayed delivery, multiplied by the weight value of the order. To comply with the properties of \textit{hub-and-spoke} networks, in which all orders are delivered in a single route, we consider a TSP for each used service $s\in \{S|destination_{s}\in L\}$, starting from $dest_{s}\in L$ and traversing through the delivery points of all orders that were shipped by service $s$.

As the objective function of each TSP includes the total transportation costs and the delay penalties, the last-mile stage is structured as a set of multiple Single-Machine Scheduling Problems with sequence-dependent times, aiming to minimize the total weighted tardiness, which are known to be strongly NP-hard \cite{Law77}.

An illustrative example of the intermodal network, the available services that connect the hubs, and the processes that are conducted in each stage of the problem is presented in Figure \ref{fig:network}.

\begin{figure}[h]
\begin{center}\includegraphics[scale=0.5]{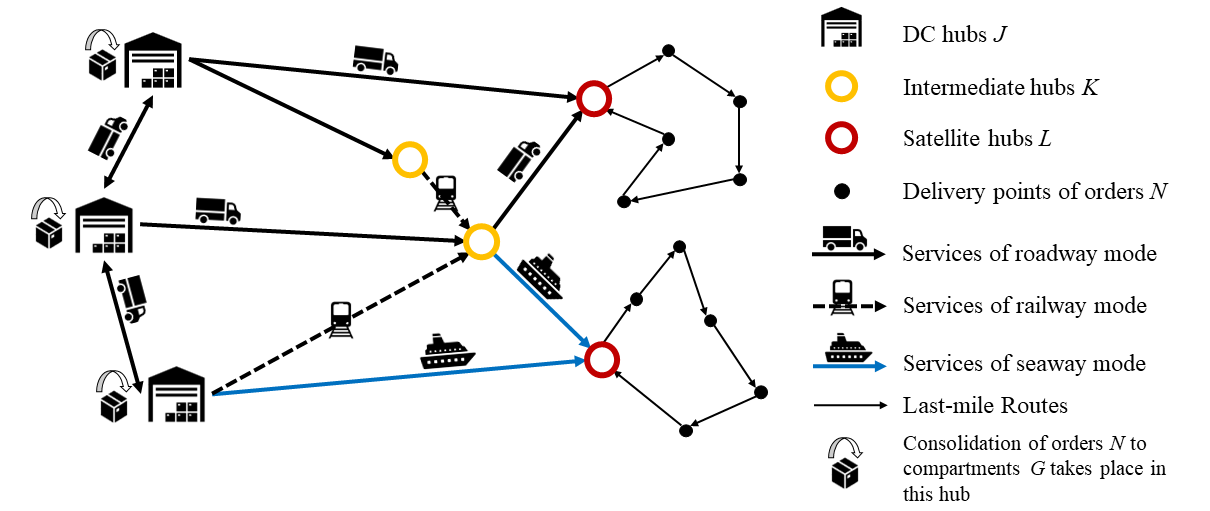}
\caption{A sketch of our network}\label{fig:network}
\end{center}
\end{figure}

\subsection{Mathematical Model}
Let $\mathcal{P}$ be an MILP formulation for the unified optimization problem that was described above. Parameters and variables of $\mathcal{P}$ are as in Table \ref{tab:notation}.

\begin{table}[H]
\centering
\caption{Sets, parameters and variables}
\resizebox{\textwidth}{!}{
\begin{tabular}{ll}
\hline
\textbf{Sets}                &\\ \hline
$I$                 & Set of nodes\\
$J$                 & Set of DC hubs, $J\subset I$\\
$K$                 & Set of Intermediate hubs, $K\subset I$\\
$L$                 & Set of Satellite hubs, $L\subset I$\\
$N$                 & Set of orders\\
$N_{j}$             & Set of orders that will be consolidated in DC hub $j\in J$, $N_{j}\subset N$\\
$G$                 & Set of loading compartments\\
$G_{j}$             & Set of compartments that are dedicated to DC hub $j\in J$, $G_{j}\subset G$\\
$S$                 & Set of scheduled intermodal services\\
$S_{L}$             & Set of services, arriving to a Satellite hub; $S_{L}=\{s\in S|dest_{s}\in L\}$\\
$\delta^+(i)$     & Set of services that depart from node $i\in I$,   $\delta^+(i)\subset S$, $origin_{s} = i$\\
$\delta^-(i)$     & Set of services that arrive to node $i\in I$,   $\delta^-(i)\subset S$, $dest_{s} = i$\\
$F_{s}$             & Set of forbidden co-assignments of service $s$, $F_{s}\subset S$\\ \hline
\multicolumn{2}{l}{\textbf{Parameters}}\\ \hline
$q$                 & Capacity of compartment $g\in G$ (identical for all compartments)\\
$d_{n}$             & Quantity of order $n\in N$\\
$t_{n}$             & Due date of order $n\in N$\\
$w_{n}$             & Weight (importance) of order $n\in N$\\
$Q_{s}$             & Number of compartments that service $s\in S$ carries at most\\
$c^{travel}_{s}$    & Travel cost per compartment for service $s\in S$\\
$c^{fixed}_{s}$     & Fixed cost of service $s\in S$\\
$departure_{s}$     & Time of departure of service $s\in S$\\
$arrival_{s}$       & Time of departure of service $s\in S$\\
$origin_{s}$        & Origin node of service $s\in S$, $origin_{s}\in I\setminus L$\\
$dest_{s}$   & Destination node of service $s\in S$, $dest_{s}\in I$\\
$M$                 & A big numeric value\\
$c_{nm}$            & Distance cost between the delivery points of $n, m\in N$\\ 
$c_{ln}$            & Distance cost between the delivery point of $n\in N$ and satellite $l\in L$\\ \hline
\textbf{Variables}          &\\ \hline
$x_{gs}$            & 1 if compartment $g$ is assigned to service $s$, 0 otherwise\\
$y_{ng}$            & 1 if order $n$ is assigned to compartment $g$, 0 otherwise\\
$h_{ns}$            & 1 if order $n$ is shipped by service $s$, 0 otherwise\\
$e_{g}$             & 1 if compartment $g\in G$ is used, 0 otherwise\\
$E_{g}$             & The quantity that is assigned to compartment $g\in G$\\
$v_{s}$             & 1 if service $s\in S$ is used, 0 otherwise\\
$V_{s}$             & The number of compartments that are assigned to service $s\in S$\\ 
$\gamma_{snm}$      & 1 if the delivery point of $n\in N$ precedes the delivery point of $m\in N$, in the last-mile route that succeeds $s\in S_{L}$, 0 otherwise\\
$\gamma^{+}_{sln}$  & 1 if the delivery point of order $n\in N$ is the first node of the last-mile route that succeeds $s\in \delta^-(l)$, 0 otherwise\\ 
$\gamma^{-}_{sln}$  & 1 if the delivery point of order $n\in N$ is the last node of the last-mile route that succeeds $s\in \delta^-(l)$, 0 otherwise\\
$C_{sn}$            & Delivery time of order $n\in N$ in the last-mile route that succeeds $s\in S_{L}$\\
$C^{max}_{s}$       & Completion time of the last-mile route that succeeds $s\in S_{L}$\\
$T_{n}$             & Number of days of delayed delivery of order $n\in N$\\ \hline      
\end{tabular}
}
\label{tab:notation}
\end{table}

The objective function of $\mathcal{P}$ is the minimum sum of intermodal transportation costs, last-mile transportation costs and delay penalties. Constraints (\ref{eq:1}) assign each order to exactly one loading compartment. Note that $y_{ng} = 0$ if order $n$ and compartment $g$ are located to different DC hubs $j\in J$. Constraints (\ref{eq:2}) define the total quantity that is loaded to each compartment, using an integer variable $E_{g}$, which is bounded by capacity $q$. Constraints (\ref{eq:3}) ensure that if any order is assigned to a compartment $g\in G$, then $g$ is used. Constraints (\ref{eq:2})-(\ref{eq:3}) are inspired by the valid inequalities of \cite{Aar98}, which replace regular capacity constraints to tighten formulations for facility location problems \cite{Ave21}.

\begingroup
\scriptsize \selectfont \setlength{\abovedisplayskip}{0pt}\setlength{\belowdisplayskip}{0pt}
\begin{flalign}
\text{$\mathcal{P}$:} \notag \\
\text{min } &\sum_{s\in S}(c^{travel}_{s}\cdot V_{s} + c^{fixed}_{s}\cdot v_{s}) + \sum_{s\in S_{L}}C^{max}_{s} + \sum_{n\in N}w_{n}\cdot T_{n} && \notag &&\\
&\text{subject to:} && \notag &&\\
&\sum_{g\in G} y_{ng} = 1 && \forall g\in G \label{eq:1}&&\\
&E_{g} = \sum_{n\in N} d_{n}\cdot y_{ng} && \forall g\in G \label{eq:2}\\
&e_{g} \geq y_{ng} && \forall g\in G, n\in N \label{eq:3}\\
&\sum_{s\in \delta^+(j)}x_{gs} = e_{g} && \forall j\in J, g\in G_{j} \label{eq:4}\\
&\sum_{s\in \delta^+(i)}x_{gs} - \sum_{s\in \delta^-(i)}x_{gs} = 0 && \forall g\in G, i\in I\setminus L\cup \{j\} = \{J|g\in G_{j}\} \label{eq:5}\\
&\sum_{l\in L}\sum_{s\in \delta^-(l)}x_{gs} = e_{g} && \forall g\in G \label{eq:6}\\
&x_{gs} + x_{gt} \leq 1 && \forall s\in S, t\in F_{s} \label{eq:7}\\
&V_{s} = \sum_{g\in G}x_{gs} && \forall s\in S \label{eq:8}\\
&v_{s} \geq x_{gs} && \forall s\in S, g\in G \label{eq:9}\\
&h_{ns} + 1 \geq x_{gs} + y_{ng} && \forall n\in N, g\in G, s\in S_{L} \label{eq:10}\\
&\sum_{n\in N}\gamma^{+}_{sln} = v_{s} && \forall l\in L, s\in \delta^-(l) \label{eq:11}\\
&\sum_{n\in N}\gamma^{-}_{sln} = v_{s} && \forall l\in L, s\in \delta^-(l) \label{eq:12}\\
&\sum_{m\in N}\gamma_{snm} + \gamma^{-}_{sln} = h_{ns} && \forall l\in L, s\in \delta^-(l), n\in N \label{eq:13}\\
&\sum_{m\in N}\gamma_{smn} + \gamma^{+}_{sln} = h_{ns} && \forall l\in L, s\in \delta^{-(l)}, n\in N \label{eq:14}\\
&\gamma^{+}_{sln} \leq h_{ns} && \forall n\in N, l\in L, s\in \delta^-(l) \label{eq:15}\\
&\gamma^{-}_{sln} \leq h_{ns} && \forall n\in N, l\in L, s\in \delta^-(l) \label{eq:16}\\
&\gamma_{snm} \leq h_{ns} && \forall n, m\in N, s\in S_{L} \label{eq:17}\\
&\gamma_{snm} \leq h_{ms} && \forall n, m\in N, s\in S_{L} \label{eq:18}\\
&C_{sn} + c_{nm} - C_{sm} \leq M\cdot (1 - \gamma_{snm}) && \forall n, m\in N, s\in S_{L} \label{eq:19}\\
&arrival_{s}\cdot h_{ns} + c_{ln} - C_{sn} \leq M\cdot (1 - \gamma^{+}_{sln}) && \forall l\in L, s\in \delta^-(l), n\in N \label{eq:20}\\
&T_{n} \geq \frac{C_{sn} - t_{n}}{24} && \forall n\in N, s\in S_{L} \label{eq:21}\\
&C^{max}_{s} = \sum_{n\in N}\sum_{m\in N\setminus n} c_{nm}\cdot \gamma_{snm} + \sum_{n\in N}c_{ln}\cdot \gamma^{+}_{sln} && \forall l\in L, s\in \delta^-(l) \label{eq:22}\\
& && \notag \\
&x_{gs} \in \{0, 1\} && \forall g\in G, s\in S \notag &&\\
&y_{ng} \in \{0, 1\} && \forall g\in G, n\in N \notag &&\\
&e_{g} \in \{0, 1\} && \forall g\in G \notag &&\\
&0\leq E_{g}\leq q, E_{g}\in \mathbb{Z} && \forall g\in G \notag &&\\
&v_{s} \in \{0, 1\} && \forall s\in S \notag &&\\
&0\leq V_{s}\leq Q_{s}, V_{s}\in \mathbb{Z} && \forall s\in S \notag &&\\
&h_{ns} \in \{0, 1\} && \forall n\in N, s\in S_{L} \notag &&\\
&\gamma_{snm}\in \{0, 1\} && \forall s\in S_{L}, n\in N, m\in N \notag &&\\
&\gamma^{+}_{sln}\in \{0, 1\} && \forall l\in L, s\in \delta^-(l), n\in N \notag &&\\
&\gamma^{-}_{sln}\in \{0, 1\} && \forall l\in L, s\in \delta^-(l), n\in N \notag &&\\
&C_{sn}\geq 0 && \forall n\in N, s\in S_{L} \notag &&\\
&C^{max}_{s}\geq 0 &&\forall s\in S_{L} \notag &&\\
&T_{n}\geq 0, T_{n}\in \mathbb{Z} &&\forall n\in N \notag &&\\
& && \notag &&
\end{flalign}
\endgroup

Constraints (\ref{eq:4})-(\ref{eq:6}) are flow conservation inequalities. Constraints (\ref{eq:4}) ensure that each used compartment will be assigned to one service that departs from its dedicated DC hub. If compartment $g$ is assigned to a service that arrives to a hub, which is not a Satellite or its dedicated DC hub, then it must be assigned to another service that departs from it, by Constraints (\ref{eq:5}). Each used compartment must be shipped to any Satellite hub (\ref{eq:6}). To ensure that the constructed intermodal routes are valid, if any compartment is assigned to a service, then it cannot be assigned to any forbidden co-assignment (\ref{eq:7}). The number of assigned compartments for each service is determined by Constraints (\ref{eq:8}), and, if any compartment is assigned to a service, then the latter is considered as used (\ref{eq:9}).

Constraints (\ref{eq:10}) ensure that if order $n$ is assigned to a compartment $g$ which is shipped by service $s$, then $n$ will be assigned to $s$ too. (\ref{eq:10}) are redundant for the intermodal transportation stage, nevertheless they are required to ensure the connectivity with the succeeding last-mile stage. By Constraints (\ref{eq:15}) and (\ref{eq:18}), we notice that each delivery point will be a part of the last-mile route that succeeds a service $s$, only if the respective order was shipped by $s$. Constraints (\ref{eq:11}) and (\ref{eq:12}) ensure that the last-mile route that succeeds used services $s$ will start from $dest_{s}\in L$ and will return to it. Constraints (\ref{eq:13}) and (\ref{eq:14}) ensure that the delivery points of all assigned orders will be a part of the route. The delivery times of orders are determined by Constraints (\ref{eq:19}), if the delivery point is an intermediate node of the route, or Constraints (\ref{eq:20}), if the delivery point is the first node of the route (after the starting satellite hub). Constraints (\ref{eq:21}) define the number of days of delay and Constraints (\ref{eq:22}) define the total transportation cost of the route.

\section{The Decomposition} \label{Section:LBBD}
LBBD \cite{Hook03, Hook07} defines a master problem that feeds a subproblem (or a group of subproblems) with a fixed solution. The solution of the subproblem(s) is an upper bound of the original optimization problem and implies the construction of inequalities, which are known as \textit{cuts}. These cuts are added to a \textit{master} problem, which provides a valid lower bound of the original problem. If the subproblem is infeasible, the \textit{feasibility cuts} restrict the master problem from computing the same fixed solution that implied the infeasibility of the subproblem. If the subproblem computes a feasible upper bound, the \textit{optimality cuts} ensure that the master problem will compute the same fixed solution only if a better one cannot be found. The solution of the master problem will be the next fixed solution that will be fed to the subproblem(s) in the next iteration, until a convergence is reached. A valid Logic-Based Benders Decomposition ensures that the lower and the upper bound will converge to the optimal solution of the original problem.

Let $\mathcal{P} = min \{f(\textsc{x}) + g(\textsc{y})|\textsc{x}\in D_{\textsc{x}}, \textsc{y}\in D_{\textsc{y}}\}$ be the formulation of an optimization problem. We consider two groups of variables $\textsc{x}$ and $\textsc{y}$ and their respective domains $D_{\textsc{x}}$ and $D_{\textsc{y}}$. $f(\textsc{x})$ and $g(\textsc{y})$ are linear cost functions that concern variables $\textsc{x}$ and $\textsc{y}$ respectively. Problem $\mathcal{P}$ is decomposed into the master problem $\mathcal{M} = min \{z|z\geq f(\textsc{x}), \textsc{x}\in D_{\textsc{x}}\}$ and the subproblem(s) $\mathcal{S} = min \{g(\textsc{y}) + f(\hat{\textsc{x}})|\textsc{y}\in D_{\textsc{y}}\}$, in which $z$ is the upper bound of function $f(\textsc{x})$ and  $\hat{\textsc{x}}$ is a feasible solution of $\mathcal{M}$.

\subsection{The master problem}
The master problem $\mathcal{M}$ formulates the intermodal transportation stage. Its objective is given by a continuous variable $z$ that is bounded from below by the sum of travel and fixed costs of used services (\ref{eq:Mobj}). We notice that $\mathcal{M}$ includes Constraints (\ref{eq:1})-(\ref{eq:10}) of $\mathcal{P}$:

\begingroup
\scriptsize \selectfont \setlength{\abovedisplayskip}{0pt}\setlength{\belowdisplayskip}{0pt}
\begin{flalign}
\text{$\mathcal{M}$:} \notag \\
\text{min } &z && \notag &&\\
&\text{subject to:} && \notag &&\\
&\sum_{g\in G} y_{ng} = 1 && \forall g\in G \notag &&\\
&E_{g} = \sum_{n\in N} d_{n}\cdot y_{ng} && \forall g\in G \notag \\
&e_{g} \geq y_{ng} && \forall g\in G, n\in N \notag \\
&\sum_{s\in \delta^+(j)}x_{gs} = e_{g} && \forall j\in J, g\in G_{j} \notag \\
&\sum_{s\in \delta^+(i)}x_{gs} - \sum_{s\in \delta^-(i)}x_{gs} = 0 && \forall g\in G, i\in I\setminus L\cup \{j\} = \{J|g\in G_{j}\} \notag \\
&\sum_{l\in L}\sum_{s\in \delta^-(l)}x_{gs} = e_{g} && \forall g\in G \notag \\
&x_{gs} + x_{gt} \leq 1 && \forall s\in S, t\in F_{s} \notag \\
&V_{s} = \sum_{g\in G}x_{gs} && \forall s\in S \notag \\
&v_{s} \geq x_{gs} && \forall s\in S, g\in G \notag \\
&h_{ns} + 1 \geq x_{gs} + y_{ng} && \forall n\in N, g\in G, s\in S_{L} \notag \\
&z \geq \sum_{s\in S}(c^{travel}_{s}\cdot V_{s} + c^{fixed}_{s}\cdot v_{s}) && \label{eq:Mobj}\\
& && \notag \\
&x_{gs} \in \{0, 1\} && \forall g\in G, s\in S \notag &&\\
&y_{ng} \in \{0, 1\} && \forall g\in G, n\in N \notag &&\\
&e_{g} \in \{0, 1\} && \forall g\in G \notag &&\\
&0\leq E_{g}\leq q, E_{g}\in \mathbb{Z} && \forall g\in G \notag &&\\
&v_{s} \in \{0, 1\} && \forall s\in S \notag &&\\
&0\leq V_{s}\leq Q_{s}, V_{s}\in \mathbb{Z} && \forall s\in S \notag &&\\
&h_{ns} \in \{0, 1\} && \forall n\in N, s\in S_{L} \notag &&\\
& && \notag &&
\end{flalign}
\endgroup

Let $\textsc{x} = \{x_{gs}, e_{g}, E_{g}, v_{s}, V_{s}, y_{ng}, h_{ns}\}$ be the group of variables of $\mathcal{M}$ and $f(\textsc{x})$ the right-hand side of (\ref{eq:Mobj}).

\subsection{A family of subproblems}
Given a solution of $\mathcal{M}$, namely $\hat{\textsc{x}} = \{\hat{x}_{gs}, \hat{e}_{g}, \hat{E}_{g}, \hat{v}_{s}, \hat{V}_{s}, \hat{y}_{ng}, \hat{h}_{ns}\}$, the subproblem $\mathcal{S}$ computes the optimal last-mile routing. As a \textit{hub-and-spoke} layout is considered, each used service $s\in \{S_{L}|\hat{v}_{s}=1\}$ is followed by a single route that starts from $dest_{s}$ and includes the delivery points of all orders $n\in N$ for which $\hat{h}_{ns} = 1$.

As each route is independent, we distinguish a different subproblem $\mathcal{S}^{s}$ for each $s\in \{S_{L}|\hat{v}_{s}=1\}$. Despite the fact that Single-Machine Scheduling problems with sequence-dependent times are NP-hard, considering separate subproblems for each route allows us to split set $N$ into smaller, tractable by exact methods that related literature offers, instances. For our decomposition scheme, we opt for a variation of the MILP that \cite{Pey20} present, adapted on a due-date related objective function of single-machine settings. Therefore, for each service $s\in \{S_{L}|\hat{v}_{s}=1\}$, let $N_{s} = \{N|\hat{h}_{ns} = 1\}$ be the set of shipped orders and $\mathcal{S}^s$ be the respective subproblem:

\begingroup
\scriptsize \selectfont \setlength{\abovedisplayskip}{0pt}\setlength{\belowdisplayskip}{0pt}
\begin{flalign}
\text{$\mathcal{S}^{s}$:} \notag \\
\text{min } &\zeta_{s} && \notag &&\\
&\text{subject to:} && \notag &&\\
&\sum_{n\in N_{s}}\gamma^{+}_{ln} = 1 && l = dest_{s} \label{eq:s1}\\
&\sum_{n\in N_{s}}\gamma^{-}_{ln} = 1 && l = dest_{s} \label{eq:s2}\\
&\sum_{m\in N_{s}}\gamma_{nm} + \gamma^{-}_{ln} = 1 && \forall n\in N_{s}, l=dest_{s} \label{eq:s3}\\
&\sum_{m\in N_{s}}\gamma_{mn} + \gamma^{+}_{ln} = 1 && l\in dest_{s}, n\in N_{s} \label{eq:s4}\\
&C_{n} + c_{nm} - C_{m} \leq M\cdot (1 - \gamma_{nm}) && \forall n, m\in N_{s} \label{eq:s5}\\
&arrival_{s} + c_{ln} - C_{n} \leq M\cdot (1 - \gamma^{+}_{ln}) && \forall n\in N_{s}, l=dest_{s} \label{eq:s6}\\
&T_{n} \geq \frac{C_{n} - t_{n}}{24} && \forall n\in N_{s} \label{eq:s7}\\
&C^{max} = \sum_{n\in N_{s}}\sum_{m\in N_{s}\setminus n} c_{nm}\cdot \gamma_{nm} + \sum_{n\in N_{s}}c_{ln}\cdot \gamma^{+}_{ln} && l = dest_{s} \label{eq:s8}\\
&\zeta_{s}\geq C^{max} + \sum_{n\in N_{s}}w_{n}\cdot T_{n} && \label{eq:Sobj} \\
& && \notag \\
&\gamma_{nm} \in \{0, 1\} && \forall n\in N_{s}, m\in N_{s} \notag &&\\
&\gamma^{+}_{ln} \in \{0, 1\} && \forall n\in N_{s}, l = dest_{s} \notag &&\\
&\gamma^{-}_{ln} \in \{0, 1\} && \forall n\in N_{s}, l = dest_{s} \notag &&\\
&C_{n}\geq 0 && \forall n\in N_{s} \notag &&\\
&T_{n}\geq 0, T_{n}\in \mathbb{Z} && \forall n\in N_{s} \notag &&\\
&C^{max}\geq 0 && \notag &&\\
& && \notag &&
\end{flalign}
\endgroup

We notice that $\mathcal{S}^{s}$ is a reduction of Constraints (\ref{eq:11})-(\ref{eq:22}) of $\mathcal{P}$ to a single-entity setting, considering fixed values for variables $h_{ns}$ and $v_{s}$ (as provided by the solution $\hat{\textsc{x}}$). We also notice that the sum of all objective values of the subproblems is equal to the total last-mile transportation and penalty costs that are involved in the objective function of $\mathcal{P}$: 

\begin{flalign}
&\sum_{s\in \{S_{L}|v_{s} = 1\}}\zeta_{s} = \sum_{s\in \{S_{L}|v_{s} = 1\}}C^{max}_{s} + \sum_{n\in N}w_{n}\cdot T_{n} && \label{eq:tc} \\
& && \notag &&
\end{flalign}

Therefore, let $\textsc{y}^s = \{\gamma_{nm}, \gamma^{+}_{ln}, \gamma^{-}_{ln}, C_{n}, T_{n}, C^{max}\}$ be the group of variables of each $\mathcal{S}^{s}$ and $g(\textsc{y}^{s}) = \zeta_{s}$. The sum of functions $f(\textsc{x})$ and $g(\textsc{y}^{s})$ is equal to the objective function of the original problem $\mathcal{P}$. 

\subsection{A Relaxation for the master problem: Partial Benders Decomposition}
Although $\mathcal{M}$ integrates the intermodal transportation stage, which is usually the most costly part of the supply chain, we notice that the objective function completely neglects the last-mile transportation costs and the delay penalties. To avoid the roll-over of excessive delay costs from the intermodal transportation stage to the last-mile routing, we implement a variation of the novel \textit{Partial Benders Decomposition} technique, to secure a more efficient performance of our decomposition scheme.

To retain information of the objective function of the subproblem to the master problem, we consider a straightforward, but nevertheless effective relaxation of the original last-mile transportation and penalty costs.

Let $c^{min}_{n}$ be the minimum distance cost of the delivery point of order $n\in N$ (i.e., $c^{min}_{n} = min\{c_{mn}, c_{ln}|l\in L, m\in N\setminus n\}$). We notice that, for each used service $s\in \{S_{L}|\hat{v}_{s}=1\}$, the transportation cost $C^{max}$ is equal with the sum of all sequence-dependent distance costs (\ref{eq:s8}). Hence, as $c^{min}_{n} \leq c_{mn}$ for all delivery points of $n, m\in N_{s}$ or Satellite hubs $l\in L$, a lower bound of the transportation cost is $C^{max} \geq \sum_{n\in N_{s}}c^{min}_{n}$. Let $C^{\star}_{s}$ be a continuous variable, defined by the following set of inequalities, which is added to $\mathcal{M}$:

\begin{flalign}
&C^{\star}_{s} \geq \sum_{n\in N}c^{min}_{n}\cdot h_{ns} && \forall s\in S_{L} \label{eq:r1}\\
& && \notag &&
\end{flalign}

Now, we attempt to compute a valid lower bound of penalty deliveries. A tight approximation of penalty costs could be decisive for the performance of the decomposition, as the optimal solution of the intermodal transportation stage tends to be in conflict with these costs. Tardiness $T$ is computed by function $T = max\{0, C - d\}$, $C$ and $d$ being the completion and due time respectively. We notice that, for all time instances $t_{1}$ and $t_{2}$ for which $t_{1} < t_{2}$, $t_{1} - d < t_{2} - d \longrightarrow T_{1} \leq T_{2}$.

Let $n\in N$ be an order that is shipped to Satellite hub $l\in L$ via service $s\in S_{L}$. The earliest possible delivery time of $n$ is the arrival time of service $s$, added by the distance between $l$ and the delivery point of $n$. Hence, a lower bound of the delivery time of order $n$ is $(arrival_{s} + c_{ln})\cdot h_{ns}$:

\begin{flalign}
&T^{\star}_{n} \geq \frac{arrival_{s} + c_{ln} - t_{n}}{24}\cdot h_{ns} && \forall n\in N, l\in L, s\in S_{L} \label{eq:r2} \\
& && \notag &&
\end{flalign}

$T^{\star}_{n}$ being an integer variable, indicating the lower bound of the actual number of days of delay. Note that the delay is counted in number of days, requiring a division of the right-hand side with 24. Last, we integrate the relaxation variables to the objective function of $\mathcal{M}$:

\begin{flalign}
&z \geq \sum_{s\in S}(c^{travel}_{s}\cdot V_{s} + c^{fixed}_{s}\cdot v_{s}) + \sum_{n\in N}w_{n}\cdot T^{\star}_{n} + \sum_{s\in S_{L}}C^{\star}_{s} && \label{eq:r3} \\
& && \notag &&
\end{flalign}

\subsection{The LBBD Algorithm}
We are now in a position to describe the method computing the optimal solution to $\mathcal{P}$. The convergence of Algorithm \ref{alg:bda} depends on the validity of the constructed \textit{optimality cuts}. The algorithm iteratively solves $\mathcal{S}^s$ for each used service $s$ that arrives to Satellite hubs, after being fed by a solution of $\mathcal{M}$. Algorithm \ref{alg:bda} finds at each iteration a solution to $\mathcal{M}$ of non-decreasing value for the objective $z$ and terminates once the computed gap is lower than a pre-defined value or a limit of iterations is reached.

\begin{algorithm}[H]
	\fontsize{9}{11}\selectfont
	\caption{\label{alg:bda} A Logic-Based Benders Decomposition Algorithm}
\begin{algorithmic}[1]
\STATE Let $r = 0$ be an iteration number;
\STATE Solve $\mathcal{M}$, including Constraints (\ref{eq:r1}) to (\ref{eq:r3}) and let $\hat{\textsc{x}} = \{\hat{x}_{gs}, \hat{e}_{g}, \hat{E}_{g}, \hat{v}_{s}, \hat{V}_{s}, \hat{y}_{ng}, \hat{h}_{ns}\}$ be its optimal solution and  $\hat{z}$ its objective value; 
\STATE Let $\hat{z}$ be the incumbent lower bound on the optimal solution of $\mathcal{P}$;
\STATE Let $\epsilon^{0}(\hat{\textsc{x}}) = 1$ or any positive number;
\WHILE {$f(\hat{\textsc{x}}) + \epsilon^{r}(\hat{\textsc{x}}) > \hat{z}$}
\STATE Set $r = r + 1$
\STATE Solve $\mathcal{S}^{s}$ for all $s\in \{S_{L}|\hat{v}_{s}=1\}$;
\STATE Let $\epsilon^{r}(\hat{\textsc{x}})$ be the sum of the objective values of $\mathcal{S}^{s}$;
\STATE Construct a bounding function $B^{r}_{s}(\hat{h}_s)$ for each  service $s\in \{S_{L}|\hat{v}_{s}=1\}$, where $\hat{h}_s$ are the values $\hat{h}_{ns}$ for all orders $n\in N_{s}$;
\STATE Add Optimality Cuts (\ref{eq:c1}), (\ref{eq:c2}) to $\mathcal{M}$;
\STATE Solve $\mathcal{M}$;
\STATE Let the solution of $\mathcal{M},$ namely $\hat{\textsc{x}}_{r},$ be the optimal solution;
\STATE Let $\hat{z}$ be the new lower bound of $\mathcal{P}$;
\STATE Set $\hat{\textsc{x}} = \hat{\textsc{x}}_{r}$;
\STATE Let $f(\hat{\textsc{x}})$ be the total cost value in $\textsc{x}$;
\ENDWHILE
\STATE return $\hat{z}$;
\end{algorithmic}
\end{algorithm}

\textbf{Optimality Cuts.} Let $\hat{\textsc{x}}_{r-1} = \{\hat{x}^{r-1}_{gs}$, $\hat{e}^{r-1}_{g}$, $\hat{E}^{r-1}_{g}$, $\hat{v}^{r-1}_{s}$, $\hat{V}^{r-1}_{s}$, $\hat{y}^{r-1}_{ng}$, $\hat{h}^{r-1}_{ns}\}$ be the solution of $\mathcal{M}$ at iteration $r-1$, and $N^{r-1}_{s}$ be the set of orders that have been transferred to a Satellite hub by a service $s\in S_{L}$: $N^{r-1}_{s} = \{n\in N|\hat{h}^{r-1}_{ns} = 1, dest_{s}\in L, \hat{v}^{r-1}_{s} = 1\}$. As $\hat{\textsc{x}}_{r-1}$ is provided into the family of subproblems of the following iteration, let $\zeta^{r}_{s}$ be the cost contribution of each subproblem $\mathcal{S}^{s}$ at iteration $r$:

\begin{flalign}
&\zeta^{r}_{s} = C^{max}_{s} + \sum_{n\in N_{s}}w_{n}\cdot T_{n} && \forall s\in \{S_{L}|\hat{v}^{r-1}_{s} = 1\} \label{eq:cc} \\
& && \notag &&
\end{flalign}

By (\ref{eq:tc}) and (\ref{eq:cc}), we construct a function $g(\textsc{y}_{r})$, which is linked with a solution $\hat{\textsc{x}}_{r-1}$ of $\mathcal{M}$:

\begin{flalign}
& \sum_{s\in \{S_{L}|\hat{v}^{r-1}_{s} = 1\}}\zeta^{r}_{s} = g(\textsc{y}_{r}) = \epsilon^{r}(\hat{\textsc{x}}_{r-1}) && \label{eq:g(y)} \\
& && \notag &&
\end{flalign}

Now, we construct a bounding function $B^{r}_{s}(\hat{h}^{r}_{s})$ for each service $s\in \{S_{L}|\hat{v}^{r-1}_{s} = 1\}$, where $\hat{h}^{r}_{s}$ are the values $\hat{h}^{r}_{ns}$ for all orders $n\in N^{r-1}_{s}$:
\begin{equation}
	B^{r}_{s}(\hat{h}_s^r) = 
\begin{cases}
	\zeta^{r}_{s} & \text{if } \{n\in N^{r-1}_{s}|\hat{h}_{ns}^{r} = 0\} = \varnothing \\
	0 & \text{otherwise} \label{eq:Bs}
	\end{cases}
\end{equation}
Roughly described, the bounding function is either equal to $\zeta^{r}_{s}$, if all orders $n\in N^{r-1}_{s}$ are assigned to the same service, or 0, if at least one order $n\in N^{r-1}_{s}$ is assigned to a different service.

The following technical lemma is crucial in order to guarantee LBBD correctness and its proof adopts the style of \cite{Hook07}.
\begin{lemma}
Let $B^{r}(\textsc{x}_{r}) = \sum_{s\in \{S_{L}|\hat{v}^{r-1}_{s} = 1\}}B^{r}_{s}(h^{r}_{s}) + f(\textsc{x}_r)$ and $\hat{\textsc{x}}_{r-1}$ be an optimal solution of $\mathcal{M}$ in iteration $r-1$. Then, the following two properties hold:
\begin{description}
	\item{$P_1$: } If $\textsc{x}_{r}= \hat{\textsc{x}}_{r-1}$, then $B^{r}(\textsc{x}_{r}) = f(\hat{\textsc{x}}_{r-1}) + \epsilon^{r}(\hat{\textsc{x}}_{r-1})$.
	\item{$P_2$: } $f(\textsc{x}'_{r-1}) + g(\textsc{y}'_r) \geq B^{r}(\textsc{x}'_{r})$, for any feasible solutions  $\textsc{x}'_{r-1}$, $\textsc{x}'_{r}$ and $\textsc{y}'_r$ in iterations $r-1$, $r$ respectively.
\end{description}\label{le:1}
\end{lemma}
\begin{proof}
For Property $P_1$, if $\textsc{x}_{r} = \hat{\textsc{x}}_{r-1}$, then $\{n\in N^{r-1}_{s}|\hat{h}^{r}_{ns} = 0\} = \varnothing$ for all $s\in \{S_{L}|\hat{v}^{r-1}_{s} = 1\}$, since all orders are assigned to the same service in iterations $r-1$ and $r$. By (\ref{eq:Bs}), $B^{r}_{s}(h^{r}_{s}) = \zeta^{r}_{s}$ $\forall s\in \{S_{L}|\hat{v}^{r-1}_{s} = 1\} \rightarrow \sum_{s\in \{S_{L}|\hat{v}^{r-1}_{s} = 1\}}B^{r}_{s}(h^{r}_{s}) = \sum_{s\in \{S_{L}|\hat{v}^{r-1}_{s} = 1\}}\zeta^{r}_{s}$. By (\ref{eq:g(y)}), $\sum_{s\in \{S_{L}|\hat{v}^{r-1}_{s} = 1\}}B^{r}_{s}(h^{r}_{s}) = \epsilon^{r}(\hat{\textsc{x}}_{r-1})$. By adding with $f(\textsc{x}_r)$, $\sum_{s\in \{S_{L}|\hat{v}^{r-1}_{s} = 1\}}B^{r}_{s}(h^{r}_{s}) + f(\textsc{x}_r) = \epsilon^{r}(\hat{\textsc{x}}_{r-1}) + f(\textsc{x}_r)$. The left-hand side is equal to $B^{r}(\textsc{x}_r)$ and the right-hand side is equal to $\epsilon^{r}(\hat{\textsc{x}}_{r-1}) + f(\hat{\textsc{x}}_{r-1})$ for $\textsc{x}_{r} = \hat{\textsc{x}}_{r-1}$.\\

For Property $P_2$, assume to the contrary that there is a feasible solution $\textsc{x}'_{r-1}$, $\textsc{y}'_{r}$ for $\mathcal{M}$ and $\mathcal{S}$ respectively, in iteration $r$, such that:
\begin{equation}
    B^{r}({\textsc{x}'_r}) > f(\textsc{x}'_{r-1}) + g(\textsc{y}'_{r}) \label{eq:contradiction}
\end{equation}
If $\textsc{x}'_{r} = \hat{\textsc{x}}_{r-1}$, by Property $P_1$, $B^{r}(\textsc{x}'_{r}) = f(\hat{\textsc{x}}_{r-1}) + \epsilon^{r}(\hat{\textsc{x}}_{r-1})$. Since the objective value of $\mathcal{S}$ is $\epsilon^{r}(\hat{\textsc{x}}_{r-1}) = g(\textsc{y}'_{r})$, we yield a contradiction of (\ref{eq:contradiction}).\\
If $\textsc{x}'_{r} \neq \hat{\textsc{x}}_{r-1}$, then ${h'}_s^r \neq \hat{h}^{r-1}_{s}$ too. Thus, there is at least one order $n\in N^{r-1}_{s}$ that has been shipped by a different service in iteration $r$, and thus $\{n\in N^{r-1}_{s}|{h'}_{ns}^{r} = 0\} \neq \varnothing$. Hence, by definition, the bounding function of $s$ will be equal to 0, $B^{r}_{s}(h'^{r}_{s}) = 0$, which is definitely less or equal than the cost of service $s$ in solution $\textsc{y}'_{r}$, equal to $\zeta^{r}_{s}$. To conclude, $B^{r}_{s}(h'^{r}_{s})$ is a lower bound of $\zeta^{r}_{s}$ for all $s\in \{S_{L}|\hat{v}^{r-1}_{s} = 1\}$. Thus, $\sum_{s\in \{S_{L}|\hat{v}^{r-1}_{s} = 1\}}B^{r}_{s}(h'^{r}_{s})$ $\leq \sum_{s\in \{S_{L}|\hat{v}^{r-1}_{s} = 1\}}\zeta^{r}_{s}$, which is equivalent with $f(\textsc{x}'_{r-1}) + \sum_{s\in \{S_{L}|\hat{v}^{r-1}_{s} = 1\}}B^{r}_{s}(h'^{r}_{s}) \leq f(\textsc{x}'_{r-1}) + g(\textsc{y}'_{r})$ and thus, $B^{r}({\textsc{x}'_r}) \leq f(\textsc{x}'_{r-1}) + g(\textsc{y}'_{r})$, contradicting (\ref{eq:contradiction}).
\end{proof}

Since the bounding function $B^{r}(\textsc{x}_{r})$, as defined by Lemma \ref{le:1}, satisfies Properties $P_1$ and $P_2$ in each iteration $r$ of Algorithm \ref{alg:bda}, and the domain of variables $\textsc{y}$ is finite, the following is proved exactly as \cite[Theorem 1]{Hook07}.
\begin{theorem}
Algorithm \ref{alg:bda} converges to the optimal solution of $\mathcal{P}$ after finitely many steps. \label{theo:1}
\end{theorem}
Now, based on the validity of the bounding function, the following inequalities can validly integrate the bounding function into the formulation of $\mathcal{M}$:

\begingroup
\selectfont \setlength{\abovedisplayskip}{0pt}\setlength{\belowdisplayskip}{0pt}
\begin{flalign}
&z^{r}_{s} \geq \zeta^{r}_{s} - \zeta^{r}_{s}\cdot(|N^{r-1}_{s}| - \sum_{n\in N^{r-1}_{s}}h_{ns}) && \forall s\in \{S_{L}|\hat{v}^{r-1}_{s} = 1\} && \label{eq:c1} \\
&z \geq \sum_{s\in \{S_{L}|\hat{v}^{r-1}_{s} = 1\}}z^{r}_{s} + f(\textsc{x}_{r}) && && \label{eq:c2}\\
& && \notag &
\end{flalign}
\endgroup

where $z^{r}_{s}$ are non-negative real variables for each service $s\in \{S_{L}|\hat{v}^{r-1}_{s} = 1\}$. Constraints (\ref{eq:c1}) ensure that if any order $n\in N^{r-1}_{s}$ is assigned to a different service in iteration $r$, then $z^{r}_{s}$ will be equal to 0. On the contrary, if all orders $n\in N^{r-1}_{s}$ are assigned to the same service, then $z^{r}_{s}$ is an upper bound of $\zeta^{r}_{s}$. Constraints (\ref{eq:c2}) ensure that the objective value of $\mathcal{M}$ is an upper bound of function $B^{r}(\textsc{x}_r)$. Therefore, adding (\ref{eq:c1}) and (\ref{eq:c2}) after each iteration of Algorithm \ref{alg:bda} guarantees convergence to the optimal solution after a finite number of steps (Theorem \ref{theo:1}).

\section{Benchmarking on Random Datasets}
We perform a set of experiments to evaluate the performance of Algorithm \ref{alg:bda} on randomly generated instances. All experiments are performed on a Linux server (4 processors, 3.3 GHz CPU, 12 GB RAM) using CPLEX 20.1 (Python API). A time limit of 900 seconds is imposed on the master problem, and the execution of Algorithm \ref{alg:bda} is restricted to 20 iterations. We also impose a primal termination, if the computed \textsc{Gap} is lower than $1\%$. As a result, Algorithm \ref{alg:bda} returns a worst-case estimation of the optimality gap. Table \ref{tab:annotations} summarizes the notation that will be used for the presentation of our results.

\begin{table}[h!]
\scriptsize \centering
\begin{tabular}{|l|l|}
\hline
\textsc{Instance}               & The number of instance\\ \hline
\textsc{Time}                   & The elapsed time of the solution (in seconds)\\ \hline
\textsc{Lower Bound}            & The lower bound of the solution\\ \hline
\textsc{Upper Bound}            & The objective value of the solution\\ \hline
\rule{0pt}{8pt}\textsc{Gap}                    & $100\cdot \frac{\textsc{Upper Bound} - \textsc{Lower Bound}}{\textsc{Upper Bound}}$ ($\%$)\\ \hline
\textsc{Milp}                   & The solution of the original problem $\mathcal{P}$\\ \hline
\textsc{Lbbd}                   & The solution of Algorithm \ref{alg:bda}\\ \hline
\end{tabular}
\caption{Results Metrics}
\label{tab:annotations}
\end{table}

We consider a set of $\{2, 3\}$ DC hubs $J$, one intermediate hub $K$, and a set of $\{2, 3\}$ Satellite hubs $L$. For each node $i\in I = J\cup K\cup L$, we define a parameter $a_{im}$, which is related with the accessibility of node $i$ by mode $m$, $m\in \{roadway, railway, seaway\}$. We assume that all nodes are accessible by roadway modes (i.e., $a_{im} = 1$ $\forall i\in I, m = roadway$). The accessibility for the rest of the available modes is decided by the generator for all nodes, except for the intermediate one, which is accessible by all modes anyway.

Based on the accessibility of the nodes of the generated network, we construct a set of scheduled services $S$. We consider a weekly timetable, in which $roadway$ services are performed on Monday, Wednesday, Friday and Sunday, $railway$ services are scheduled on Tuesday, Thursday, Saturday, and $seaway$ services are scheduled on Monday, Wednesday and Friday. Departure times $departure_{s}$ are predefined, but all capacities $Q_{s}$, arrival times $arrival_{s}$, and costs $c^{travel}_{s}$ and $c^{fixed}_{s}$ are determined by the generator. For each pair of nodes $(i, j)$, $i\in J\cup K$, $j\in I$, and for all days of the week on which a scheduled service is involved, we generate a service of $origin_{s} = i$, $destination_{s} = j$, if both nodes are accessible by the same mode. Last, the subsets of forbidden co-assignments $F_{s}$ are constructed, according to the rules that were described in Section \ref{Section:LBBD}.

Now, we construct a set of $\{60, 80, 100, 120, 140, 160\}$ orders $N$. All parameters that are related to $N$ are generated by random uniforms, and each order is assigned to exactly one DC hub $j\in J$. For each assigned order, a new compartment $g\in G$ is added to the same DC hub, to ensure that the number of compartments is infinite in practice. The capacity of compartments is pre-defined.

For the last-mile stage, we define a graph of $X\in [0, 20]$ and $Y\in [0, 20]$ coordinates. For all orders $n\in N$ and Satellite hubs $l\in L$, we randomly compute a pair of coordinates. To compute the distance costs $c_{nm}$ and $c_{ln}$, we used the Euclidean distances between all involved nodes, multiplied by a randomly generated cost per kilometer value. A listing of the generators that we used for the construction of the instances is presented in Table \ref{tab:generator}. \texttt{Uniform}($\alpha, \beta$) and \texttt{Integer}($\alpha, \beta$) are random uniforms that sample continuous and integer values in range $[\alpha, \beta]$ respectively. To combine all distinct values of sets $J$, $L$ and $N$, a total of $|J|\cdot |L|\cdot |N| = 24$ instances is generated.

\begin{table}[H]
\centering
\caption{List of random generators}
\resizebox{\textwidth}{!}{
\begin{tabular}{ll}
\hline
\textbf{Sets}       &\textbf{Values}\\ \hline
$J$                 & \{2, 3\}\\
$K$                 & \{1\}\\
$L$                 & \{2, 3\}\\
$N$                 & \{60, 80, 100, 120, 140, 160\}\\
$G$                 & Equal to $N$\\
$N_{j}$             & Randomly assignment of each $n\in N$ to one $j\in J$\\ \hline
\textbf{Parameters} &\\ \hline
$a_{im}$ &\texttt{Integer}(0, 1) $\forall i\in J\cup K\cup L, m\in \{$roadway, railway, seaway$\}$\\
$q$                 & 5000\\
$d_{n}$             &\texttt{Integer}(1000, 5000)\\
$t_{n}$             & 24$\cdot$ \texttt{Integer}(7, 10)\\
$w_{n}$             & \texttt{Integer}(1, 10)\\
$Q_{s}$             & \texttt{Integer}(2, 4) for roadway, \texttt{Integer}(3, 6) for railway, \texttt{Integer}(5, 10) for seaway\\
$c^{travel}_{s}$    & $(arrival_{s} - departure_{s})\cdot \alpha$, $\alpha = 1.0$ for roadway, $0.8$ for railway, $0.6$ for seaway\\
$c^{fixed}_{s}$     & \texttt{Integer}(200, 300) for roadway, \texttt{Integer}(300, 500) for railway, \texttt{Integer}(500, 700) for seaway\\
$departure_{s}$     & $24\cdot d + 9, d\in \{0, 2, 4, 6\}$ for roadway, $24\cdot d + 13, d\in \{1, 3, 5\}$ for railway, $24\cdot d + 10, d\in \{0, 2, 4\}$ for seaway\\
$arrival_{s}$       & $departure_{s} + $ $24\cdot $\texttt{Integer}(0, 3) $+ 21$ for $J$, $24\cdot $\texttt{Integer}(1, 3) $+ 29$ for $K$, $24\cdot $\texttt{Integer}(2, 5) $+ 33$ for $L$ destinations\\
$X_{n}, Y_{n}$      & \texttt{Uniform}(0, 20) for all $n\in N\cup L$\\
$c_{nm}$            & Euclidean distance $\cdot$ \texttt{Uniform}(0.1, 1.0)\\ \hline
\end{tabular}
}
\label{tab:generator}
\end{table}

The results of the MILP of $\mathcal{P}$ and Algorithm \ref{alg:bda} (denoted by \textsc{Milp} and \textsc{Lbbd} respectively) are summarized in Table \ref{tab:results}. The number of intermodal services $|S|$ is unstable, as we introduced the stochastic parameter $a_{im}$. More specifically, a random uniform determines whether hub $i$ is accessible by services of mode $m\in \{$roadway, railway, seaway$\}$. If $a_{im} = a_{jm} = 1$ for a pair of hubs $i, j$, then we generate a new batch of services $s|origin_{s} = i, dest_{s} = j$. To ensure the feasibility of the generated instances, we assume that all hubs are accessible by roadway services (i.e., $a_{im} = 1$, $\forall i\in J\cup K\cup L$, $m =$ roadway). In addition, the intermediate hub $K$ is accessible by all modes, so that all orders can possibly be shipped to any Satellite hub $l\in L$. The generated scheduled timetables for all 24 instances are available at \url{https://github.com/i-avgerinos/LBBD_Intermodal}.

Regarding the $\textsc{Milp}$, we impose a timelimit of 30000 seconds. If $\textsc{Time} < 30,000$ and $\textsc{Gap} > 1\%$, then the solver returned an out-of-memory error before the timelimit was reached. If the solver did not provide any feasible solution before reaching the timelimit or an out-of-memory status, then the values of \textsc{Upper Bound} and \textsc{Gap} are void.

\begin{table}[H]
\centering
\resizebox{\textwidth}{!}{
\begin{tabular}{|c|c|c|c|c|cccc|cccc|}
\hline
\multirow{2}{*}{\textsc{Instance}} & \multirow{2}{*}{$|J|$} & \multirow{2}{*}{$|L|$} & \multirow{2}{*}{$|N|$} & \multirow{2}{*}{$|S|$} & \multicolumn{4}{c|}{\textsc{Milp}}                                                                                                        & \multicolumn{4}{c|}{\textsc{Lbbd}}                                                                                                        \\ \cline{6-13} 
                                   &                        &                        &                        &                        & \multicolumn{1}{c|}{\textsc{Time}} & \multicolumn{1}{c|}{\textsc{Lower Bound}} & \multicolumn{1}{c|}{\textsc{Upper Bound}} & \textsc{Gap} & \multicolumn{1}{c|}{\textsc{Time}} & \multicolumn{1}{c|}{\textsc{Lower Bound}} & \multicolumn{1}{c|}{\textsc{Upper Bound}} & \textsc{Gap} \\ \hline
1                                  & 2                      & 2                      & 60                     & 78                     & \multicolumn{1}{c|}{30,000}         & \multicolumn{1}{c|}{6,451}                 & \multicolumn{1}{c|}{10,032}                & 29.02        & \multicolumn{1}{c|}{5,056}          & \multicolumn{1}{c|}{7,121}               & \multicolumn{1}{c|}{7,179}              & 0.80         \\ \hline
2                                  & 2                      & 3                      & 60                     & 93                     & \multicolumn{1}{c|}{30,000}         & \multicolumn{1}{c|}{5,280}                 & \multicolumn{1}{c|}{5,619}                 & 5.10         & \multicolumn{1}{c|}{4,432}          & \multicolumn{1}{c|}{5,332}              & \multicolumn{1}{c|}{5,373}              & 0.75         \\ \hline
3                                  & 3                      & 2                      & 60                     & 116                    & \multicolumn{1}{c|}{30,000}         & \multicolumn{1}{c|}{6,997}                 & \multicolumn{1}{c|}{7,434}                 & 3.39         & \multicolumn{1}{c|}{3,605}          & \multicolumn{1}{c|}{7,182}               & \multicolumn{1}{c|}{7,252}              & 0.97         \\ \hline
4                                  & 3                      & 3                      & 60                     & 144                    & \multicolumn{1}{c|}{30,000}         & \multicolumn{1}{c|}{5,939}                 & \multicolumn{1}{c|}{6,062}                 & 1.02         & \multicolumn{1}{c|}{664}           & \multicolumn{1}{c|}{6,000}              & \multicolumn{1}{c|}{6,051}               & 0.83         \\ \hline
5                                  & 2                      & 2                      & 80                     & 72                     & \multicolumn{1}{c|}{30,000}         & \multicolumn{1}{c|}{8,632}                & \multicolumn{1}{c|}{-}                    & -            & \multicolumn{1}{c|}{11,315}         & \multicolumn{1}{c|}{9,616}              & \multicolumn{1}{c|}{9,690}              & 0.77         \\ \hline
6                                  & 2                      & 3                      & 80                     & 78                     & \multicolumn{1}{c|}{5,514}          & \multicolumn{1}{c|}{8,092}                 & \multicolumn{1}{c|}{-}                    & -            & \multicolumn{1}{c|}{18,104}         & \multicolumn{1}{c|}{8,791}              & \multicolumn{1}{c|}{9,151}              & 3.93         \\ \hline
7                                  & 3                      & 2                      & 80                     & 104                    & \multicolumn{1}{c|}{30,000}         & \multicolumn{1}{c|}{9,175}                 & \multicolumn{1}{c|}{11,328}                & 10.00        & \multicolumn{1}{c|}{20,147}         & \multicolumn{1}{c|}{10,195}             & \multicolumn{1}{c|}{11,057}             & 7.79         \\ \hline
8                                  & 3                      & 3                      & 80                     & 161                    & \multicolumn{1}{c|}{3,255}          & \multicolumn{1}{c|}{6,925}                 & \multicolumn{1}{c|}{-}                    & -            & \multicolumn{1}{c|}{22,261}         & \multicolumn{1}{c|}{8,073}              & \multicolumn{1}{c|}{8,434}                 & 4.28         \\ \hline
9                                  & 2                      & 2                      & 100                    & 87                     & \multicolumn{1}{c|}{30,000}         & \multicolumn{1}{c|}{8,931}                & \multicolumn{1}{c|}{-}                    & -            & \multicolumn{1}{c|}{25,104}         & \multicolumn{1}{c|}{9,718}              & \multicolumn{1}{c|}{9,999}              & 2.82         \\ \hline
10                                 & 2                      & 3                      & 100                    & 75                     & \multicolumn{1}{c|}{30,000}         & \multicolumn{1}{c|}{9,446}                 & \multicolumn{1}{c|}{-}                    & -            & \multicolumn{1}{c|}{18,022}         & \multicolumn{1}{c|}{18,045}             & \multicolumn{1}{c|}{19,429}             & 7.13         \\ \hline
11                                 & 3                      & 2                      & 100                    & 119                    & \multicolumn{1}{c|}{-}             & \multicolumn{1}{c|}{-}                    & \multicolumn{1}{c|}{-}                    & -            & \multicolumn{1}{c|}{18,032}         & \multicolumn{1}{c|}{13,494}             & \multicolumn{1}{c|}{14,174}             & 4.80         \\ \hline
12                                 & 3                      & 3                      & 100                    & 135                    & \multicolumn{1}{c|}{3,002}          & \multicolumn{1}{c|}{9,428}                 & \multicolumn{1}{c|}{-}                    & -            & \multicolumn{1}{c|}{20,448}         & \multicolumn{1}{c|}{11,106}             & \multicolumn{1}{c|}{11,433}             & 2.86         \\ \hline
13                                 & 2                      & 2                      & 120                    & 93                     & \multicolumn{1}{c|}{-}             & \multicolumn{1}{c|}{-}                    & \multicolumn{1}{c|}{-}                    & -            & \multicolumn{1}{c|}{25,237}         & \multicolumn{1}{c|}{13,796}             & \multicolumn{1}{c|}{15,329}             & 10.00        \\ \hline
14                                 & 2                      & 3                      & 120                    & 90                     & \multicolumn{1}{c|}{-}             & \multicolumn{1}{c|}{-}                    & \multicolumn{1}{c|}{-}                    & -            & \multicolumn{1}{c|}{18,042}         & \multicolumn{1}{c|}{14,817}             & \multicolumn{1}{c|}{15,671}             & 5.45         \\ \hline
15                                 & 3                      & 2                      & 120                    & 119                    & \multicolumn{1}{c|}{-}             & \multicolumn{1}{c|}{-}                    & \multicolumn{1}{c|}{-}                    & -            & \multicolumn{1}{c|}{20548}         & \multicolumn{1}{c|}{13,212}             & \multicolumn{1}{c|}{14,803}             & 10.75        \\ \hline
16                                 & 3                      & 3                      & 120                    & 155                    & \multicolumn{1}{c|}{-}             & \multicolumn{1}{c|}{-}                    & \multicolumn{1}{c|}{-}                    & -            & \multicolumn{1}{c|}{20,542}         & \multicolumn{1}{c|}{11,539}              & \multicolumn{1}{c|}{11,914}             & 3.15         \\ \hline
17                                 & 2                      & 2                      & 140                    & 90                     & \multicolumn{1}{c|}{-}             & \multicolumn{1}{c|}{-}                    & \multicolumn{1}{c|}{-}                    & -            & \multicolumn{1}{c|}{28,040}         & \multicolumn{1}{c|}{12,690}             & \multicolumn{1}{c|}{13,285}             & 4.48         \\ \hline
18                                 & 2                      & 3                      & 140                    & 123                    & \multicolumn{1}{c|}{-}             & \multicolumn{1}{c|}{-}                    & \multicolumn{1}{c|}{-}                    & -            & \multicolumn{1}{c|}{26,032}         & \multicolumn{1}{c|}{13,640}             & \multicolumn{1}{c|}{14,229}             & 4.14         \\ \hline
19                                 & 3                      & 2                      & 140                    & 104                    & \multicolumn{1}{c|}{-}             & \multicolumn{1}{c|}{-}                    & \multicolumn{1}{c|}{-}                    & -            & \multicolumn{1}{c|}{18,904}         & \multicolumn{1}{c|}{20,756}             & \multicolumn{1}{c|}{23,049}             & 9.95         \\ \hline
20                                 & 3                      & 3                      & 140                    & 183                    & \multicolumn{1}{c|}{-}             & \multicolumn{1}{c|}{-}                    & \multicolumn{1}{c|}{-}                    & -            & \multicolumn{1}{c|}{29,963}         & \multicolumn{1}{c|}{12,395}             & \multicolumn{1}{c|}{12,916}             & 4.03         \\ \hline
21                                 & 2                      & 2                      & 160                    & 78                     & \multicolumn{1}{c|}{-}             & \multicolumn{1}{c|}{-}                    & \multicolumn{1}{c|}{-}                    & -            & \multicolumn{1}{c|}{21,473}         & \multicolumn{1}{c|}{33,150}             & \multicolumn{1}{c|}{36,415}             & 8.97         \\ \hline
22                                 & 2                      & 3                      & 160                    & 108                    & \multicolumn{1}{c|}{-}             & \multicolumn{1}{c|}{-}                    & \multicolumn{1}{c|}{-}                    & -            & \multicolumn{1}{c|}{20,542}         & \multicolumn{1}{c|}{28,467}             & \multicolumn{1}{c|}{34,437}             & 17.34        \\ \hline
23                                 & 3                      & 2                      & 160                    & 107                    & \multicolumn{1}{c|}{-}             & \multicolumn{1}{c|}{-}                    & \multicolumn{1}{c|}{-}                    & -            & \multicolumn{1}{c|}{20,543}         & \multicolumn{1}{c|}{21,767}             & \multicolumn{1}{c|}{24,263}             & 10.29        \\ \hline
24                                 & 3                      & 3                      & 160                    & 126                    & \multicolumn{1}{c|}{-}             & \multicolumn{1}{c|}{-}                    & \multicolumn{1}{c|}{-}                    & -            & \multicolumn{1}{c|}{20,451}         & \multicolumn{1}{c|}{21,827}             & \multicolumn{1}{c|}{23,639}             & 7.66         \\ \hline
\end{tabular}
}
\caption{Comparison of \textsc{Lbbd} with \textsc{Milp}}
\label{tab:results}
\end{table}

For the \textsc{Lbbd}, we notice that, for the smaller instances (up to 100 orders), the computed \textsc{Gaps} are steadily much lower than 5\% (with an exception for \textsc{Instance} 7 and 10). The five smallest instances manage to reach the primal termination limit of $1\%$ \textsc{Gap}, while, for the rest of them, the maximum allowable number of iterations is consumed. We notice that, despite the increasing of the number of orders, the respective number of services does not analogously rise. This is a conscious choice, as we are more interested in evaluating the performance of our algorithm, if the available services are not adequate to ship all orders in time, rather than estimating the limitations of our formulations in terms of numbers of variables and constraints. Even though the performance of \textsc{Lbbd} declines as the number of orders increases, only a small number of instances provides \textsc{Gap} of more than $10\%$.

Nevertheless, \textsc{Lbbd} outperforms the holistic \textsc{Milp} under any circumstance. \textsc{Milp} provides feasible solutions for only 5 out of 24 instances, although showing a good performance for three of them. For all instances, though, \textsc{Lbbd} provides good feasible solutions (improving the respective solutions of \textsc{Milp}, if computed) in significantly less time than \textsc{Milp}. We also notice that the computed \textsc{Lower Bounds} of \textsc{Lbbd} are tighter.

\section{A Case-Study Extension} \label{section:Case-study}
Motivated by a real logistics supply chain setting of a major European 3PL provider, we extend the methodological framework that we introduced above to show its versatility on actual problems. We exploit the fact that the practical restrictions of the provider allow us to reduce the domains of feasible solutions and integrate additional processes of the supply chain in our unified approach. Then, we implement our extended optimizing method on several real weekly instances of considerable scale.

\subsection{Case-Study Description}
The business case that has motivated this paper has been provided by EKOL (\url{https://www.ekol.com/}, a major 3PL logistics company operating mainly in Europe, as part of a funded research project. We describe the setting in abstract terms and periodically explain the Ekol-specific values of various parameters, always remaining consistent with the terminology and notation that Section \ref{Section:Model} describes.

Let us introduce a first-mile stage, which includes the collection of orders $N$ from pickup points that are located to Poland, Czechia, Hungary and Slovakia. Each order $n\in N$ is linked with a pickup point and a customs clearance office, apart from the parameters that we presented in Table \ref{tab:notation}. All collected orders are unloaded to a set of three DC hubs $J = \{$Katowice, Ostrava, Budapest$\}$. Before the orders are transferred to their dedicated hub, though, a customs clearance process must be conducted in the respective customs clearance office. The time-span of all loading, customs clearance and unloading processes is restricted by the operating hours of all companies involved. Each DC hub is served by a finite number of single-compartment identical trucks.

A few orders suffice to fill a compartment, due to tight capacities of the latter, all companies have quite strict operating hours and each pickup node is dedicated to a single customs company and a single DC hub. All these aspects reduce the number of first-mile routes, so as we can enumerate them, even for instances of significant size. Hence, the associated VRPTW can be formulated as an exact-covering problem, where columns are in one-to-one correspondence with routes and must `cover' the rows which are in one-to-one correspondence with pickup points. Let $S^{-}$ be a subset of the already defined set of scheduled services $S$, including all predefined feasible routes, and $\alpha_{ns}$ be a binary parameter, indicating whether the pickup point of order $n\in N$ is served by route $s\in S^{-}$. The subset of intermodal scheduled services is renamed to $S^{+}$.

The intermodal transportation stage of the problem is as described in Section \ref{Section:Model}. We consider an intermediate hub in the port of Trieste (i.e., $K = \{$Trieste$\}$) and two Satellite hubs in Turkey: $L = \{$Istanbul, Mersin$\}$. A fixed timetable of weekly scheduled services connects the nodes of the intermodal network. All nodes are connected with roadway services, conducted by multi-compartment trucks. The DC hub of Ostrava is linked with the port of Trieste by rail and the hub in Budapest is connected with the rail station of Istanbul. Last, the intermediate hub is the origin of seaway services to the ports of Istanbul and Mersin.

For the last-mile stage, we consider two alternative delivery processes: the \textit{hub-and-spoke} and the \textit{connected-hubs} layouts. As the first one does not differ from the process that we described previously, let us focus on the latter case. In a \textit{connected-hubs} network, the compartments that arrive to Satellite hubs $L$ are unloaded, and their components (i.e., the shipped orders), are consolidated to new vehicles, which complete their delivery. As this stage is also subject to similar capacities, time windows and resource constraints as the first mile, we pre-compute routes as in the first-mile stage and apply the set-covering formulation for the associated VRPTW. In addition, given the tight capacities of the new vehicles, the short time slots of the operating hours of the delivery points and the fact that occupying vehicles for long routes will increase the delayed deliveries, we can safely assume that each vehicle delivers at most three orders per route. To distinguish the orders that will be delivered by a \textit{hub-and-spoke} or a \textit{connected-hubs} process, each service that arrives to Satellite hubs is duplicated. Combining all alternatives with the different departure times and network layouts implies a large number of scheduled services. Overall, the intermodal network of EKOL has 143 roadway services, 22 railway services and 30 seaway services per week, which connect three DC hubs, two Satellite hubs (and their \textit{connected-hubs} copies) and one intermediate hub.

The objective function is added by the transportation costs of the newly presented first-mile stage. The delay penalties are computed by the product of the squared number of days of delay with the weight of importance of each order, enhancing the contribution of tardiness on the overall cost. An illustrative example of the supply chain of the presented case is shown in Figure \ref{fig:real}.

\begin{figure}[h]
\begin{center}\includegraphics[scale=0.5]{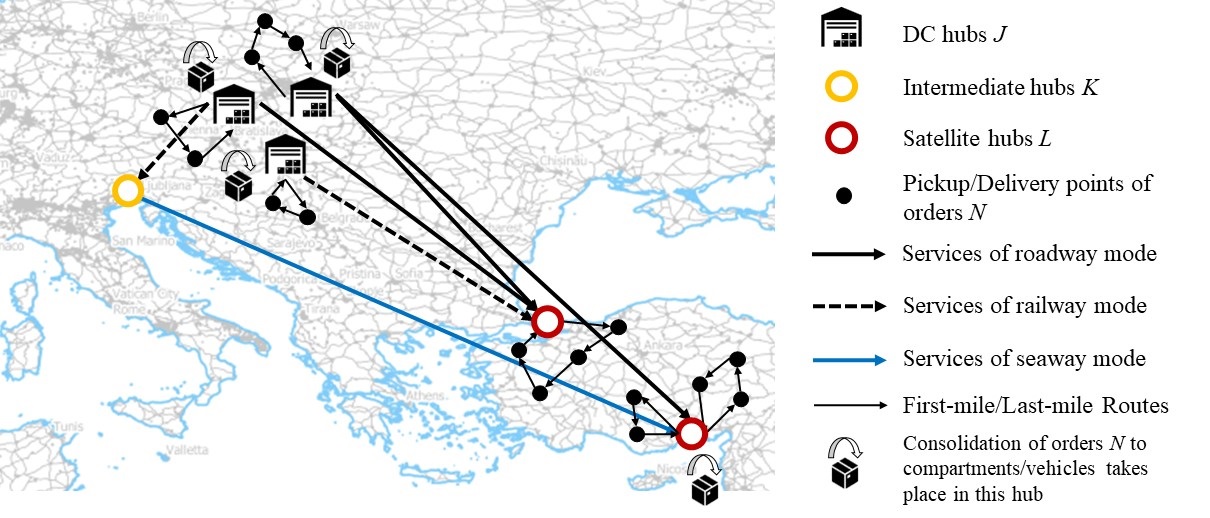}
\caption{An indicative intermodal network for the Case-Study}\label{fig:real}
\end{center}
\end{figure}

\subsection{Extending the LBBD for a three-stages Supply Chain}
As several new processes are added on the supply chain, the formulations of the master problem and the subproblems must be reformed. Constraints (\ref{eq:1})-(\ref{eq:10}), though, remain intact, as the intermodal stage of the case-study is as described for the generic version of the problem. Also, the relaxation of the last-mile transportation costs (\ref{eq:r1}) is not affected, as it is still valid for the \textit{connected-hubs} option. Note, though, that the annotation of $S$ is replaced with $S^{+}$ for all involved constraints, as the set of services $S$ now includes the first-mile routes $S^{-}$ too.\\

\textbf{Master problem.} To integrate the first-mile stage in the master problem, we add the following inequalities:

\begingroup
\selectfont \setlength{\abovedisplayskip}{0pt}\setlength{\belowdisplayskip}{0pt}
\begin{flalign}
&\sum_{s\in S^{-}}\alpha_{ns}\cdot f_{s} = 1 && \forall n\in N && \label{eq:uc1}\\
&f_{s}\in \{0, 1\} && \forall s\in S^{-} \notag \\
& && \notag &
\end{flalign}
\endgroup

Variables $f_{s}$ indicate whether route $s\in S^{-}$ is conducted. Constraints (\ref{eq:uc1}) are regular exact-covering inequalities, extensively used for routing problems \cite{San11, San13}. We notice that no resource constraints are added to $\mathcal{M}$, despite the fact that the fleet of vehicles is limited. To avoid increasing the complexity of the already overloaded master problem, we attempt to address the resources restrictions to the subproblems, assuming an infinite number of vehicles for $\mathcal{M}$.\\

\textbf{Subproblems for the \textit{Connected-Hubs} layout.} The subproblems $\mathcal{S}^{s}$ of Section \ref{Section:LBBD} are still valid, if the service $s$ that shipped orders $N_{s}$ is compatible with the \textit{hub-and-spoke} layout. For the \textit{connected-hubs} case, though, the splitability of the subproblems per used service is invalid, as orders of different shipments can be consolidated into the same delivery route. As we did for the first-mile routes, let $R$ be a set of pre-defined last-mile routes and $\beta_{nr}$ be a binary parameter that determines whether the delivery point of order $n$ is served by $r\in R$.

Following a preliminary computational experimentation among possible formulations\cite{Pey20}, a CP formulation is suggested as the most efficient. This is because CP can handle resource-allocation constraints more easily through so-called `global functions', thus also avoiding the use of time-indexed variables. In addition, a CP formulation can handle the first-mile resources restrictions that were rolled over by $\mathcal{M}$ to $\mathcal{S}$, barely adding any complexity.

We define the subproblem $\mathcal{S}^{c}$:

\begingroup
\scriptsize \selectfont \setlength{\abovedisplayskip}{0pt}\setlength{\belowdisplayskip}{0pt}
\begin{flalign}
\text{$\mathcal{S}^c$:} \notag \\
\text{min } &\sum_{r\in R}\texttt{length\_of}(\phi_{r}) + \sum_{n\in N}w_{n}\cdot \tau_{n} && \notag &&\\
&\text{subject to:} && \notag &&\\
&\texttt{Cumulative}((\texttt{start\_of}(\pi_{s})|origin_{s} = j), (\texttt{size\_of}(\pi_{s})|origin_{s}=j), 1, K_{j}) &&\forall j\in J \label{eq:uc2}&&\\
&\alpha_{ns} = 1 \rightarrow \texttt{end\_of}(\pi_{s}) \leq shipment_{n} && \forall n\in N, s\in \{S^{-}|\hat{f}_{s} = 1\} \label{eq:uc3}&&\\
&\sum_{r\in R}\beta_{nr}\cdot \texttt{presence\_of}(\phi_r) = 1 && \forall n\in N && \label{eq:uc4} \\
&\texttt{presence\_of}(\phi_{r}) = \texttt{presence\_of}(o_{r_{i}}) && \forall r\in R, i\in \{1, |r|\} && \label{eq:uc5} \\
&\texttt{start\_of}(\phi_{r}) = \texttt{start\_of}(o_{r_{0}}) && \forall r\in R && \label{eq:uc6} \\
&\texttt{end\_of}(\phi_{r}) = \texttt{end\_of}(o_{r_{|r|}}) && \forall r\in R && \label{eq:uc7} \\
&\texttt{presence\_of}(o_{r_{i}})\rightarrow \texttt{end\_of}(o_{r_{i-1}}) \leq \texttt{start\_of}(o_{r_{i}}) && \forall r\in R, i\in \{2, |r|\} && \label{eq:uc8} \\
&\texttt{presence\_of}(o_{r_{i}}) \rightarrow T_{r_{i}} \geq \frac{\texttt{end\_of}(o_{r_{i}}) - t_{r_{i}}}{24} && \forall r\in R, i\in \{2, |r|\} && \label{eq:uc9} \\
&\tau_{n} = T_{n}^2 && \forall n\in N && \label{eq:uc10} \\
&\texttt{Cumulative}((\texttt{start\_of}(\phi_{r})), (\texttt{length\_of}(\phi_{r})), 1, K_{l}) &&\forall l\in L \label{eq:uc11}&&\\
& && \notag \\
&\pi_{s} : \texttt{intervalVar}([0, M], [0, M], duration_{s}) && \forall s\in \{S^-|\hat{f}_{s} = 1\} \notag &&\\
&\phi_{r} : \texttt{intervalVar}([t^{route}_{r}, M], [t^{route}_{r}, M], \texttt{optional}) && \forall r\in R \notag &&\\
&o_{r_{i}} : \texttt{intervalVar}([t^{route}_{r}, M], [t^{route}_{r}, M], c_{r_{i-1}r_{i}}, \texttt{optional}) && \forall r\in R, i\in \{2, |r|\} \notag &&\\
&T_{n} \in \{0, M\} && \forall n\in N \notag &&\\
&\tau_{n} \in \{0, M\} && \forall n\in N \notag &&\\
& && \notag &&
\end{flalign}
\endgroup

Constraints (\ref{eq:uc2}) and (\ref{eq:uc3}) address to the resource-allocation problem of the first-mile routes. Let $\{s\in S^-|\hat{f}_{s} = 1\}$ be the set of selected first-mile routes of the solution of the master problem and $\pi_{s}$ be interval variables, starting at $\texttt{start\_of}(\pi_{s})\in [0, M]$ and ending at $\texttt{end\_of}(\pi_{s})\in [0, M]$, after $\texttt{size\_of}(\pi_{s}) = duration_{s} = arrival_{s} - departure_{s}$. Considering a limited number of $K_{j}$ trucks for DC hub $j$, we apply the $\texttt{Cumulative}$ predicate (\ref{eq:uc2}) \cite{Cumul93}, to ensure that no more than $K_{j}$ routes $s$ will be conducted simultaneously. By (\ref{eq:uc3}), if $shipment_{n}$ is the departure time of the intermodal service that ships order $n$ from the DC hub $j$, the arrival time of the first-mile route that serves $n$ must precede.

The interval variables $\phi_{r}$ indicate the start time and the end time of route $r$, if that route is selected, assuming that $t^{route}_{r}$ is the latest arrival time of the services that shipped the involved orders. The start time of each interval variable $o_{r_{i}}$ indicates the departure time from node $i-1$ and their end time is the arrival time to node $i$. Variables $T_{n}$ and $\tau_{n}$ define the number of days of delayed delivery of order $n$ and its squared value, respectively. $M$ is a big number, used for an upper bound of the domains of variables, and $K_{l}$ is the number of available last-mile routing vehicles in Satellite hub $l\in L$.

Constraints (\ref{eq:uc4}) ensure that each order will be served by exactly one route. Constraints (\ref{eq:uc5}) ensure that if the interval variable of a route $r$ is present, then all interval variables of the nodes of the route will be present too. Constraints (\ref{eq:uc6}) and (\ref{eq:uc7}) ensure that the start time of route $r$ is equal to the departure from the hub and the end time is equal with the arrival to the hub. Moreover, Constraints (\ref{eq:uc8}) ensure that if $o_{r_{i}}$ is present, then the arrival to node $i-1$ will be precedent of the departure from $i-1$ to serve node $i$. Constraints (\ref{eq:uc9}) and (\ref{eq:uc10}) define the number of days of delayed delivery and the squared value respectively. Finally, we consider a \texttt{Cumulative} constraint in (\ref{eq:uc11}), to ensure that all vehicles are scheduled, considering their availability.

As the objective value of each subproblem $\mathcal{S}^{s}$ is denoted by $\zeta_s$, let $\zeta_{l}$ be the cost contribution of orders $N^{r-1}_{l}$ to the objective value of $\mathcal{S}^c$.\\

\textbf{Reformulations of \textit{hub-and-spoke} subproblems.} To integrate the nonlinear penalty cost function, we reformulate variables $T_{n}$ and the constraints in which they are involved in $\mathcal{S}^{s}$. As the cost function is quadratic, we opt for a regular linearization technique, using binary variables. Let $\tau_{nd}$ be equal to 1 if order $n\in N$ is delivered $d$ days later than its deadline $t_{n}$. For each order $n$, only one variable $\tau_{nd}$ if fixed to 1:

\begingroup
\selectfont \setlength{\abovedisplayskip}{0pt}\setlength{\belowdisplayskip}{0pt}
\begin{flalign}
&\sum_{d = 0, ..., M}\tau_{nd} = 1 && \forall n\in N && \label{eq:uc100} \\
&\tau_{nd}\in \{0, 1\} && \forall n\in N, d = 0, ..., M \notag &&\\
& && \notag &
\end{flalign}
\endgroup

To ensure that $\tau_{nd}$ will be equal to 1 only if $d$ is greater than the delay of $n$, we replace (\ref{eq:s7}) with:

\begingroup
\selectfont \setlength{\abovedisplayskip}{0pt}\setlength{\belowdisplayskip}{0pt}
\begin{flalign}
&\sum_{d = 0, ..., M}d\cdot \tau_{nd} \geq \frac{C_{n} - t_{n}}{24} && \forall n\in N && \label{eq:uc101} \\
& && \notag &
\end{flalign}
\endgroup

Last, the objective function of $\mathcal{S}^s$ considers the nonlinear costs, by replacing (\ref{eq:Sobj}) with:

\begingroup
\selectfont \setlength{\abovedisplayskip}{0pt}\setlength{\belowdisplayskip}{0pt}
\begin{flalign}
&\zeta_{s}\geq C^{max} + \sum_{n\in N_{s}}\sum_{d = 0, ...,M}w_{n}\cdot d^{2}\cdot \tau_{nd} &&  && \label{eq:uc102} \\
& && \notag &
\end{flalign}
\endgroup

Note that $M$ indicates a sufficiently large number, so that no infeasibilities occur.\\

\textbf{Feasibility and Optimality Cuts.} Before we describe the modifications of cuts (\ref{eq:c1})-(\ref{eq:c2}), let us note that the formulation of $\mathcal{S}^{c}$ may imply infeasibilities for a solution of $\mathcal{M}$. If the number of available trucks in DC hubs is not sufficient for the selected first-mile routes, Constraints (\ref{eq:uc2}) yield an infeasibility.

If this is the case, a set of inequalities must secure that this solution will not be repeated in the next iterations of the algorithm, thus forcing feasibility. Let us describe these cuts.

Using the notation of \cite{Nutmeg20}, let $1_{\hat{\textsc{x}}}^{r-1}$ be the set of variables $f_{s}^r$ and $h_{ns}^r$ that receive value $1$ when solving $\mathcal{M}$ in iteration $r-1$: $1_{\hat{\textsc{x}}}^{r-1} = \{h_{ns}^{r}, f_{s}^{r}|\hat{h}_{ns}^{r-1} = 1, \hat{f}_{s}^{r-1} = 1\}$. A valid feasibility cut would force $\mathcal{M}$ to prohibit the assignment of at least one variable of $1_{\hat{\textsc{x}}}^{r-1}$. If $|1_{\hat{\textsc{x}}}^{r-1}|$ is the length of the set, the cut has the following form (see also \cite{Hook07} and examples in \cite{Nutmeg20}):

\begin{flalign}
\sum_{s\in 1_{\hat{\textsc{x}}}^{r-1}}f_{s}^{r} + \sum_{(n, s)\in 1_{\hat{\textsc{x}}}^{r-1}} h_{ns}^{r} \leq |1_{\hat{\textsc{x}}}^{r-1}| - 1 && \label{eq:fc}
\end{flalign}

The iterative addition of ($\ref{eq:fc}$) to $\mathcal{M}$, if $\mathcal{S}^{c}$ is infeasible, will lead to either a solution of $\mathcal{M}$ that implies a feasible solution of $\mathcal{S}^{c}$ or an infeasibility of $\mathcal{M}$ too.

As for the Optimality cuts, we keep on using (\ref{eq:c1}) for all used services that are followed by \textit{hub-and-spoke} deliveries. In addition, we construct a bounding function for each Satellite hub $l\in L$, which is the destination of services that precede \textit{connected-hubs} deliveries:

\begin{equation}
	B^{r}_{l}(\hat{h}_l^r) = 
\begin{cases}
	\zeta^{r}_{l} & \text{if } \{n\in N^{r-1}_{l}|\hat{h}_{ns^{r-1}_{n}}^{r} = 0\} = \varnothing \\
	0 & \text{otherwise} \label{eq:Bl}
	\end{cases}
\end{equation}
$s^{r-1}_{n}$ being the service in $\delta^-(l)$ for which $\hat{h}^{r-1}_{ns^{r-1}_{n}} = 1$, and $N^{r-1}_{l} = \{n\in N|dest_{s^{r-1}_{n}} = l\}$. As Properties $P_{1}$ and $P_{2}$ hold for a global bounding function $B^{r}(\textsc{x}_{r}) = \sum_{s\in \{S_{L}|\hat{v}^{r-1}_{s} = 1\}}B^{r}_{s}(h^{r}_{s}) + \sum_{l\in L}B^{r}_{l}(h^{r}_{l}) + f(\textsc{x}_r)$, Theorem \ref{theo:1} is valid for the following optimality cuts:

\begingroup
\selectfont \setlength{\abovedisplayskip}{0pt}\setlength{\belowdisplayskip}{0pt}
\begin{flalign}
&z^{r}_{l} \geq \zeta^{r}_{l} - \zeta^{r}_{l}\cdot(|N^{r-1}_{l}| - \sum_{n\in N^{r-1}_{l}}h_{ns^{r-1}_{n}}) && \forall l\in L && \label{eq:c3} \\
&z \geq \sum_{s\in \{S_{L}|\hat{v}^{r-1}_{s} = 1\}}z^{r}_{s} + \sum_{l\in L}z^{r}_{l} + f(\textsc{x}_{r}) && && \label{eq:c4}\\
& && \notag &
\end{flalign}
\endgroup

Cut (\ref{eq:c4}) replaces inequality (\ref{eq:c2}) at each iteration $r$, hence the lower bound $z$ will be equal to the upper bound, if a previously computed solution is globally optimal.\\

\textbf{A new relaxation for the master problem.} As mentioned before, the lower bound of the last-mile transportation costs is still valid for \textit{connected-hubs} deliveries. However, as the cost function of delay penalties differs, relaxation (\ref{eq:r2}) must be reformed.

Let $T_{n}$ be the number of days of delay and $\tau_{n} = T_{n}^{2}$. As the penalty costs are equal to $w_{n} \cdot \tau_{n}$, we propose the following substitute of (\ref{eq:r2}):

\begingroup
\selectfont \setlength{\abovedisplayskip}{0pt}\setlength{\belowdisplayskip}{0pt}
\begin{flalign}
& \tau^{\star}_{n} \geq \Bigl\lceil{\frac{arrival_{s} - t_{n}}{24}}\Bigr\rceil \cdot \Bigl\lceil{|\frac{arrival_{s} - t_{n}}{24}|}\Bigr\rceil \cdot h_{ns} && \forall n\in N, s\in \{S^{+}|dest_{s}\in L\} \label{eq:r4} &&\\
& && \notag &&
\end{flalign}
\endgroup

and the respective reformulation of (\ref{eq:r3}):

\begin{flalign}
&z \geq \sum_{s\in S}(c^{travel}_{s}\cdot V_{s} + c^{fixed}_{s}\cdot v_{s}) + \sum_{n\in N}w_{n}\cdot \tau^{\star}_{n} + \sum_{s\in S_{L}}C^{\star}_{s} && \label{eq:r5} \\
& && \notag &&
\end{flalign}

Now, to eliminate a range of solutions that will obviously conflict with constraints (\ref{eq:uc2}) and (\ref{eq:uc3}) of $\mathcal{S}^c$, implying an infeasibility, we also consider the following set of inequalities for $\mathcal{M}$:

\begingroup
\selectfont \setlength{\abovedisplayskip}{0pt}\setlength{\belowdisplayskip}{0pt}
\begin{flalign}
&\sum_{s_{2}\in \{\delta^+(j^{n})|departure_{s_{2}} \geq arrival_{s_{1}}\}}x_{g{s_2}} + 1\geq \alpha_{ns_{1}}\cdot f_{s_{1}} + y_{ng} && \forall n\in N, s_{1}\in S^{-}, g\in G_{j^n} && \label{eq:r6}\\
& && \notag &&
\end{flalign}
\endgroup

Let $j^{n}\in J$ be the dedicated DC hub of order $n\in N$. Cut (\ref{eq:r6}) secures the succession of the first-mile and the intermodal transportation stages, assuming that the number of first-mile routing vehicles is infinite (i.e., the arrival times of routes in $S^-$ are fixed). If order $n$ is collected by route $s_{1}\in S^-$, unloaded to $j^{n}$ at $arrival_{s_{1}}$, and assigned to compartment $g\in G_{j^n}$, then $g$ must be shipped by a service that departs later than $arrival_{s_{1}}$. (\ref{eq:r6}) are in synergy with constraints (\ref{eq:3}) and (\ref{eq:4}), to ensure that $\sum_{s_{2}\in \{\delta^+(j^{n})|departure_{s_{2}} \geq arrival_{s_{1}}\}}x_{g{s_2}} \leq 1$.\\

\textbf{Modifying Algorithm \ref{alg:bda}.} Most of Algorithm \ref{alg:bda} remains intact. Nevertheless, we induce a feasibility check of solutions $\hat{\textsc{x}}$, by solving subproblem $\mathcal{S}^{c}$ before $\mathcal{S}^{s}$, so that, if the subproblem is infeasible, no time will be wasted on redundantly solving parts of it. If $\mathcal{S}^{c}$ is infeasible, the optimality cuts generation process is held off, and the induction of a feasibility cut (\ref{eq:fc}) to $\mathcal{M}$ is prioritized.

If the master problem is infeasible (i.e., all feasible solutions implied infeasibilities to subproblems at previous iterations), then problem $\mathcal{P}$ is globally infeasible and the Algorithm terminates, without returning any solutions.\\

\textbf{Experimentation of real instances.} We implement the extension of Algorithm \ref{alg:bda} on several real datasets of the 3PL provider. As a monolithic MILP model which integrates all three stages cannot find even a feasible solution on any available instance, the results of LBBD are compared with the costs that the actual routing of the logistics provider planned. As the complexity of the newly added subproblem does not allow a fast solution, we impose a timelimit of 900 seconds for $\mathcal{S}^c$ too. The early termination conditions (i.e., 900 seconds for the master problem, a limit of 20 iterations, a limit of $1\%$ for the value of \textsc{Gap}) still hold.

Each service is featured with a fixed cost and a cost per compartment, which is equal to the duration of the service in hours. The fixed cost $c^{fixed}_{s}$ lays in the range of \euro$417 - 2179$. Roadway services can carry $Q_{s} = 20$ compartments at most, railway services carry 30 compartments at the time and seaway services carry either 60 or 100 compartments per route. Each compartment has a fixed capacity of $q = 27000$ payweight units (a case-specific unit, related with the volume, the weight, and the loading meter of orders). Roadway transport can be carried out by either single-driver' or double-drivers' vehicles. In the former case, the driver takes a break of $9$ hours after a shift of $9$ consecutive hours of driving; thus the route can be performed without breaks only for double-driver vehicles since each driver is resting during the other one's shift. For the first-mile routes, the number of available single-compartment trucks $K_j$ is $19$ for the DC hubs in Katowice and Ostrava and $35$ for the hub in Budapest. The Satellite hubs are served by $K_l =$ 30 vehicles, that carry out the last-mile routes. Last, the importance weight $w_{n}$ of each order $n\in N$ is between $1$ and $50$, multiplied by the cost of \euro100.

All orders that were submitted from Monday to Sunday are concentrated into a single instance. The routing of these orders begins on Monday of the next week. Each weekly instance is split into $168$ time slots, one per hour of that week. The number of orders $|N|$ per instance ranges from $41$ to $142$. A set of $|S^+| = 195$ inter-regional transport services includes a fixed timetable of routes that concerns a period of $10$ days after the week of the given instance. Each order of this dataset has an average deadline of $15$ days to be delivered.

\begin{table}[h!]
\centering
{
\begin{tabular}{|c|r|r|r|r|}
\hline
\textsc{Instance} & \textsc{Time} & \textsc{Lower Bound} & \textsc{Upper Bound} & \textsc{Gap} \\ \hline
1                 & 31,600        & 9,379                 & 9,684                 & 3.15\%         \\ \hline
2                 & 27,920        & 5,896                 & 6,030                 & 2.23\%        \\ \hline
3                 & 15,140        & 12,140                & 12,416                & 2.22\%         \\ \hline
4                 & 32,480        & 7,096                 & 7,366                 & 3.66\%         \\ \hline
5                 & 25,500        & 6,496                 & 6,778                 & 4.17\%         \\ \hline
6                 & 17,420        & 6,234                 & 6,702                 & 6.98\%         \\ \hline
7                 & 28,680        & 6,391                 & 6,457                 & 1.02\%         \\ \hline
8                 & 24,900        & 6,213                 & 6,407                 & 3.03\%         \\ \hline
9                 & 15,640        & 6,651                 & 6,767                 & 1.71\%         \\ \hline
10                & 27,480        & 6,908                 & 7,067                 & 2.26\%         \\ \hline
11                & 21,460        & 9,392                 & 11,129                & 15.61\%        \\ \hline
12                & 29,140        & 6,396                 & 6,943                 & 7.87\%         \\ \hline
\end{tabular}
}
\caption{Results of the Extended Algorithm \ref{alg:bda} for the real instances}
\label{tab:real_results}
\end{table}

Table \ref{tab:real_results} presents the performance of our method. We notice that the computed \textsc{Gaps} are steadily much lower than 10\% (with an exception for the $11^{th}$ instance). Compared to the actual solution that was implemented, the solution of Algorithm \ref{alg:bda} implies increased first-mile costs, however, the remarkable decrease of the cost for the intermodal transportation stage and the shortening of the tardiness cost lead to an improved solution, as presented in Table \ref{tab:comparison}.

Thus, we conclude that our approach can solve instances of realistic size close to optimality and indeed improves current practice, which implements a continuous consolidation and imminent shipment of collected orders (at least to the best of our understanding). Regardless of how the current planning is implemented, our approach can solve weekly instances close to optimality in reasonable time, thus being appropriate for periodic re-planning or any kind of more versatile decision support.

\begin{table}[h!]
\centering
{
\begin{tabular}{|l|r|r|r|r|r|}
\hline
                & First-mile & Intermodal Transportation & Last-mile & Penalty & Total \\ \hline
Extended Algorithm \ref{alg:bda}     & 15,932   & \textbf{64,801}     & 2,056   & \textbf{6,402} & \textbf{89,191}   \\ \hline
Actual Solution & \textbf{6,437}   & 81,544   & \textbf{1,902}    & 8,500 & 98,383  \\ \hline
\end{tabular}
}
\caption{Cost Splitting Elaboration}
\label{tab:comparison}
\end{table}

\section{Conclusions}
Motivated by the significant impact of delays on logistics and the lack of novel exact optimizing methods in literature for the unified supply chain optimization problem, we attempted to fill up these gaps by proposing a Logic-Based Benders Decomposition and proving its valid convergence to the exact optimal solution using a variety of cuts, relaxations and subproblems. Despite the complexities that the integration of the last-mile stage and the due-date related objectives adds on, the proposed method performs well on instances of comparable size with the ones presented in relevant literature, most of them being restricted to the intermodal transportation stage only.

The applicability of our method under real conditions is evaluated by an extended variant of our generic framework, implemented on a real case of a 3PL provider. This extension also integrates the first-mile stage of the Supply Chain, as alternative, more complex, last-mile delivery scenarios are also considered. Although the performance of the presented extension relies on case-specific assumptions that significantly reduce the scale of the original three-stages problem, the imposed restrictions are still realistic and reasonable. The computation of improved solutions than the actual routings that were planned in practice shows that unified approaches to optimize the entire supply chain are viable for future research.

The effect of stochasticity on route planning is an important aspect that has motivated several recent studies (e.g., \cite{Mul21,Zw21}). The main obstacle here is that realistic datasets even just for the middle-mile cannot be handled by exact methods, thus \cite{Mul21,Zw21} offer efficient heuristics. As our study focuses on exact optimizing methods, it appears that coping with stochasticity is an important, yet appealing, challenge.
\linebreak

\textbf{Acknowledgments. }This research has been supported by the EU through the COG-LO Horizon 2020 project (\url{http://www.cog-lo.eu/}), grant number 769141.

\newpage
\end{document}